\documentclass[11pt,english]{amsart}
\usepackage[T1]{fontenc}
\usepackage[latin9]{inputenc}
\usepackage{amsthm}
\usepackage{amstext}
\usepackage{amssymb}
\usepackage{esint}

\makeatletter
\numberwithin{equation}{section}
\numberwithin{figure}{section}
\theoremstyle{plain}
\newtheorem{thm}{\protect\theoremname}[section]
  \theoremstyle{plain}
  \newtheorem{lem}[thm]{\protect\lemmaname}
  \theoremstyle{plain}
  \newtheorem{prop}[thm]{\protect\propositionname}

\usepackage{amsmath,amsfonts,amssymb,amsthm,epsfig}

\usepackage{color}

\usepackage[final,allcolors=blue,colorlinks=true]{hyperref}
\usepackage[leqno]{amsmath}

\def\R{\mathbb R}

\def\wq{\infty}
\def\cal{\mathcal}
\def\al{\alpha}
\def\be{\beta}
\def\ga{\gamma}
\def\de{\delta}
\def\ep{\epsilon}
\def\la{\lambda}

\def\var{\varphi}
\def\om{\omega}
\def\na{\nabla}
\def\Ga{\Gamma}
\def\Om{\Omega}
\def\De{\Delta}

\def\pa{\partial}

\numberwithin{equation}{section}
\theoremstyle{definition}
\begingroup

\endgroup

\makeatother

\usepackage{babel}
  \providecommand{\lemmaname}{Lemma}
  \providecommand{\propositionname}{Proposition}
\providecommand{\theoremname}{Theorem}

\begin{document}

\title{\title[Uniqueness Of Solutions To  Quasilinear Equations] \,Uniqueness
of Positive Radial Solutions To Singular Critical Growth Quasilinear
Elliptic Equations }

\author{Cheng-Jun He and Chang-Lin Xiang}

\address{[Cheng-Jun He] Wuhan Institute of Physics and Mathematics, Chinese
Academy of Sciences, P.O. Box 71010, Wuhan, 430071, P. R. China.}

\email{[Cheng-Jun He] cjhe@wipm.ac.cn}

\address{[Chang-Lin Xiang] University of Jyvaskyla, Department of Mathematics
and Statistics, P.O. Box 35, FI-40014 University of Jyvaskyla, Finland.}

\email{[Chang-Lin Xiang] Xiang\_math@126.com}

\thanks{This is part of the PhD work of the second named author at the University
of Jyv\"askyl\"a.}
\begin{abstract}
In this paper, we prove that there exists at most one positive radial
weak solution to the following quasilinear elliptic equation with
singular critical growth
\[
\begin{cases}
-\Delta_{p}u-{\displaystyle \frac{\mu}{|x|^{p}}|u|^{p-2}u}{\displaystyle =\frac{|u|^{\frac{(N-s)p}{N-p}-2}u}{|x|^{s}}}+\la|u|^{p-2}u & \text{in }B,\\
u=0 & \text{on }\pa B,
\end{cases}
\]
where $B$ is an open finite ball in $\R^{N}$ centered at the origin,
$1<p<N$, $-\wq<\mu<((N-p)/p)^{p}$, $0\le s<p$ and $\la\in\R$.
A related limiting problem is also considered.
\end{abstract}

\maketitle
{\small
\noindent {\bf Keywords:} Quasilinear elliptic equations; Singular critical growth; Positive radial solutions; Pohozaev identity; Uniqueness; Asymptotic behaviors
\smallskip
\newline\noindent {\bf 2010 Mathematics Subject Classification: } 35A24; 35B33; 35B40; 35J75; 35J92

\tableofcontents{}

\section{Introduction and main results}

In this paper, we consider the following quasilinear elliptic equation
\begin{equation}
\begin{cases}
-\Delta_{p}u-{\displaystyle \frac{\mu}{|x|^{p}}|u|^{p-2}u}={\displaystyle \frac{|u|^{p^{*}(s)-2}u}{|x|^{s}}}+\la|u|^{p-2}u & \text{in }B,\\
u=0 & \text{on }\pa B,
\end{cases}\label{eq: Local Object}
\end{equation}
where $B$ is an open finite ball in $\R^{N}$ centered at the origin,
$1<p<N$, $-\wq<\mu<\bar{\mu}=((N-p)/p)^{p}$, $0\le s<p$, $p^{*}(s)=(N-s)p/(N-p)$,
$\la\in\R$ and
\begin{eqnarray*}
\Delta_{p}u=\sum_{i=1}^{N}\partial_{x_{i}}(|\nabla u|^{p-2}\partial_{x_{i}}u), &  & \nabla u=(\partial_{x_{1}}u,\cdots,\partial_{x_{N}}u),
\end{eqnarray*}
is the $p$-Laplacian operator.

It is well known that equation (\ref{eq: Local Object}) is the Euler-Lagrange
equation of the energy functional $J:W_{0}^{1,p}(B)\to\R$ defined
as
\begin{eqnarray*}
J(u)=\frac{1}{p}\int_{B}\left(|\na u|^{p}-\frac{\mu}{|x|^{p}}|u|^{p}-\la|u|^{p}\right)dx-\frac{1}{p^{*}(s)}\int_{B}\frac{|u|^{p^{*}(s)}}{|x|^{s}}dx, &  & u\in W_{0}^{1,p}(B),
\end{eqnarray*}
where the Sobolev space $W_{0}^{1,p}(B)$ is the completion of $C_{0}^{\wq}(B)$,
the space of smooth functions with compact support in $B$, in the
seminorm $\|u\|_{W_{0}^{1,p}(B)}=\|\na u\|_{L^{p}(B)}$. All the integrals
in functional $J$ are well defined, due to the Hardy inequality \cite{Garcia-Peral1998,Hardy-Littlewood-Polya}
\begin{eqnarray*}
\left(\frac{N-p}{p}\right)^{p}\int_{B}\frac{|\varphi|^{p}}{|x|^{p}}dx\leq\int_{B}|\nabla\varphi|^{p}dx, &  & \forall\:\varphi\in W_{0}^{1,p}(B),
\end{eqnarray*}
and due to the Caffarelli-Kohn-Nirenberg inequality \cite{Caffarelli-Kohn-Nirenberg1984}
\begin{eqnarray*}
C\left(\int_{B}\frac{|\var|^{p^{*}(s)}}{|x|^{s}}dx\right)^{\frac{p}{p^{*}(s)}}\leq\int_{B}|\nabla\varphi|^{p}dx, &  & \forall\:\varphi\in W_{0}^{1,p}(B),
\end{eqnarray*}
where $C=C(N,p,s)>0$.

We say that a function $u\in W_{0}^{1,p}(B)$ is a weak solution to
equation (\ref{eq: Local Object}), if for all functions $\var\in C_{0}^{\wq}(B)$,
we have
\begin{equation}
\int_{B}\left(|\nabla u|^{p-2}\nabla u\cdot\nabla\varphi-\frac{\mu}{|x|^{p}}|u|^{p-2}u\varphi-\la|u|^{p-2}u\varphi\right)dx=\int_{B}\frac{|u|^{p^{*}(s)-2}u}{|x|^{s}}\varphi dx.\label{eq: definition of weak solutions}
\end{equation}
In the following, we will systematically omit the word ``weak''
and simply say that $u$ is a solution to equation (\ref{eq: Local Object}),
meaning (\ref{eq: definition of weak solutions}); a similar convention
for weak solutions to equations in the below.

By Theorem 1.1 of Han \cite{HanPG2005}, we have the following existence
result: Consider equation (\ref{eq: Local Object}) with $s=0$, that
is,
\begin{equation}
\begin{cases}
-\Delta_{p}u-{\displaystyle \frac{\mu}{|x|^{p}}|u|^{p-2}u}=|u|^{p^{*}-2}u+\la|u|^{p-2}u & \text{in }B,\\
u=0 & \text{on }\pa B,
\end{cases}\label{eq: Han}
\end{equation}
where we write $p^{*}=p^{*}(0)$ for simplicity throughout the paper.
Assume that $1<p^{2}<N$ and $0<\mu\le N^{p-1}(N-p^{2})/p^{p}$. Then
for every $\la,$ $0<\la<\la_{1}(\mu)$, there exists at least one
positive solution to equation (\ref{eq: Han}), where $\la_{1}(\mu)$
is defined by
\begin{equation}
\la_{1}(\mu)=\inf_{u\in W_{0}^{1,p}(B)\backslash\{0\}}\frac{\int_{B}\left(|\nabla u|^{p}-\mu|x|^{-p}|u|^{p}\right)dx}{\int_{B}|u|^{p}dx}.\label{eq: first eigenvalue}
\end{equation}
In the case $p=2$, above existence result was also obtained by Jannelli
\cite{Jannelli1999} on more general domains. For more results on
existence of solutions to equation (\ref{eq: Local Object}) and its
variants, we refer to e.g. \cite{BrezisNirenberg,CaoHan2004,CaoPnegYan2012,CaoYan2010,Guedda-Veron1989,Ghoussoub-Yuan2000,HanPG2005,Jannelli1999}.

A very important ingredient in the argument of Han \cite{HanPG2005}
is the following result, which was obtained by Boumediene, Veronica
and Peral \cite{B-V-P2006}: Denote by $\mathcal{D}^{1,p}(\mathbb{R}^{N})$
the completion of $C_{0}^{\wq}(\R^{N})$, the space of smooth functions
in $\R^{N}$ with compact support, in the seminorm $\ensuremath{\|v\|_{\mathcal{D}^{1,p}(\mathbb{R}^{N})}=\|\nabla v\|_{L^{p}(\mathbb{R}^{N})}.}$
Consider the limiting problem
\begin{eqnarray}
-\Delta_{p}u-{\displaystyle \frac{\mu}{|x|^{p}}|u|^{p-2}u}={\displaystyle |u|^{p^{*}-2}u} &  & \text{in }\R^{N},\label{eq: B-V-P-1}
\end{eqnarray}
where $0<\mu<\bar{\mu}$. There is a unique ground state $U\in{\cal D}^{1,p}(\R^{N})$
to equation (\ref{eq: B-V-P-1}), up to a dilation $U^{\tau}=\tau^{-(N-p)/p}U(\cdot/\tau)$,
$\tau>0$. Moreover, $U$ is a positive radial function which satisfies
\begin{eqnarray}
\lim_{|x|\rightarrow0}U(x)|x|^{\gamma_{1}}=C_{1} & \text{and} & \lim_{|x|\rightarrow\infty}U(x)|x|^{\gamma_{2}}=C_{2},\label{est: optimal result of B-V-P-1}
\end{eqnarray}
and
\begin{eqnarray}
\lim_{|x|\rightarrow0}|\nabla U(x)||x|^{\gamma_{1}+1}=|\ga_{1}|C_{1} & \text{and} & \lim_{|x|\rightarrow\infty}|\nabla U(x)||x|^{\gamma_{2}+1}=\ga_{2}C_{2},\label{est: optimal result of B-V-P-2}
\end{eqnarray}
where $C_{1},C_{2}>0$ are constants depending only $N,p$ and $\mu$.

In the estimates (\ref{est: optimal result of B-V-P-1}) and (\ref{est: optimal result of B-V-P-2}),
the exponents $\ga_{1},\ga_{2}$ are defined as follows: Define $\Ga_{\mu}:\R\to\R$
by
\begin{eqnarray}
\Ga_{\mu}(\ga)=(p-1)|\ga|^{p}-(N-p)|\gamma|^{p-2}\ga+\mu, &  & \ga\in\R.\label{eq: definition of Gamma-mu}
\end{eqnarray}
Consider the equation
\begin{eqnarray}
\Ga_{\mu}(\ga)=0, &  & \ga\in\R.\label{eq: definition of ga-1 and ga-2}
\end{eqnarray}
Due to our assumptions on $N,p$ and $\mu$, that is, $1<p<N$ and
$-\wq<\mu<\bar{\mu}=\left((N-p)/p\right)^{p},$ equation (\ref{eq: definition of ga-1 and ga-2})
admits two and only two solutions, denoted by $\ga_{1}$ and $\ga_{2}$,
with $\ga_{1}<\ga_{2}$.

For later use, we note that in the case $0<\mu<\bar{\mu}$, we have
\[
0<\ga_{1}<\frac{N-p}{p}<\ga_{2}<\frac{N-p}{p-1},
\]
and in the case $\mu<0$, we have
\[
\ga_{1}<0<\frac{N-p}{p-1}<\ga_{2}.
\]
In the case $\mu=0$, we have
\begin{eqnarray*}
\gamma_{1}=0 & \text{and} & \ga_{2}=\frac{N-p}{p-1},
\end{eqnarray*}
and in the case  $p=2$, we have
\begin{eqnarray*}
\gamma_{1}=\sqrt{\overline{\mu}}-\sqrt{\overline{\mu}-\mu} & \text{and} & \ga_{2}=\sqrt{\overline{\mu}}+\sqrt{\overline{\mu}-\mu}.
\end{eqnarray*}

A natural question is whether the positive solution obtained by Han
\cite{HanPG2005} to equation (\ref{eq: Han}) is unique. In the case
when $p=2$, the answer is affirmative, see Ramaswamy and Santra \cite{Ramaswamy2013},
where a more general uniqueness result was obtained. In the general
case when $1<p<N$, this question has not yet been fully understood.
In this paper, we give a partial answer to this question. We have
the following uniqueness result.
\begin{thm}
\label{thm: Uniquess to Han} Assume that $1<p^{2}<N$, $0<\mu\le N^{p-1}(N-p^{2})/p^{p}$
and $0<\la<\la_{1}(\mu)$. Then equation (\ref{eq: Han}) admits at
most one positive radial solution in $B$.
\end{thm}
We remark that in the case $p=2$, positive solutions to equation
(\ref{eq: Han}) with $0<\mu<\bar{\mu}$ and $0<\la$ are radial by
Lemma 3.1 of Ramaswamy and Santra \cite{Ramaswamy2013}, while in
the general case $1<p<N$, $p\ne$2, the symmetry of positive solutions
to equation (\ref{eq: Han}) with $0<\mu<\bar{\mu}$ and $0<\la$
seems to be unknown.

Note that the result of Theorem \ref{thm: Uniquess to Han} dose not
cover the full range of the parameters $p,\mu$ and $\la$. In this
paper, we will prove the uniqueness of positive radial weak solutions
to equation (\ref{eq: Local Object}) in the full range of parameters
of $p,\mu,s$ and $\la$, that is, $1<p<N$, $-\wq<\mu<\bar{\mu}=((N-p)/p)^{p}$,
$0\le s<p$ and $\la\in\R$. The motivation for us to consider the
full range of these parameters is due to the fact that, different
ranges of these parameters have been considered extensively in the
literature. Examples will be given in the below.

We also consider the following limiting problem
\begin{eqnarray}
-\Delta_{p}u-{\displaystyle \frac{\mu}{|x|^{p}}|u|^{p-2}u}={\displaystyle \frac{|u|^{p^{*}(s)-2}u}{|x|^{s}}} &  & \text{in }\R^{N}.\label{eq: Global Object}
\end{eqnarray}
It is easy to see that equation (\ref{eq: B-V-P-1}) is a special
case of equation (\ref{eq: Global Object}). A function $u\in{\cal D}^{1,p}(\R^{N})$
is a (weak) solution to equation (\ref{eq: Global Object}), if for
every $\var\in C_{0}^{\wq}(\R^{N})$, we have
\[
\int_{\mathbb{R}^{N}}\left(|\nabla u|^{p-2}\nabla u\cdot\nabla\varphi-\frac{\mu}{|x|^{p}}|u|^{p-2}u\varphi\right)dx=\int_{\mathbb{R}^{N}}\frac{|u|^{p^{*}(s)-2}u}{|x|^{s}}\varphi dx.
\]
We will only consider positive radial weak solutions to equation (\ref{eq: Global Object}).
In the following we discuss positive radial weak solutions to equation
(\ref{eq: Local Object}) and equation (\ref{eq: Global Object})
respectively.

\subsection{Uniqueness of positive radial weak solutions to equation (\ref{eq: Local Object})}

In this subsection, we consider equation (\ref{eq: Local Object}).
We are concerned with the uniqueness of positive radial solutions
to equation (\ref{eq: Local Object}). Uniqueness problems have been
considered extensively in the literature. We refer the reader to e.g.
\cite{AdimurthiYadava1994,F-L-Serrin-2009,Kwong1989,Ramaswamy2013,Srikanth1993,ZhangLQ1992},
where more general nonlinear elliptic equations were studied.

When $p=2$, equation (\ref{eq: Local Object}) is reduced to
\begin{equation}
\begin{cases}
-\Delta u-{\displaystyle \frac{\mu}{|x|^{2}}u}={\displaystyle \frac{|u|^{2^{*}-2}u}{|x|^{s}}}+\la u & \text{in }B,\\
u=0 & \text{on }\pa B.
\end{cases}\label{eq: Ramaswamy-Santra}
\end{equation}
When $\la\le0$, it is standard to prove that equation (\ref{eq: Ramaswamy-Santra})
admits no positive solution by Pohozaev identity \cite{Pohozaev-1965}.
When $0\le\mu<\bar{\mu}=(N-2)^{2}/4$, $s=0$ and $\la>0$, it is
well known \cite{GNN,Ramaswamy2013} that positive solutions to equation
(\ref{eq: Ramaswamy-Santra}) are radial. When $\mu=s=0$ and $\la>0$,
the uniqueness of positive solutions to equation (\ref{eq: Ramaswamy-Santra})
was proved by Zhang \cite{ZhangLQ1992} and Srikanth \cite{Srikanth1993},
while for $0<\mu<\bar{\mu}=(N-2)^{2}/4$, $s=0$ and $\la>0$, the
uniqueness for positive solutions to equation (\ref{eq: Ramaswamy-Santra})
was proved by Ramaswamy and Santra \cite{Ramaswamy2013}. The ideas
of \cite{Ramaswamy2013,Srikanth1993,ZhangLQ1992} are to prove that
positive radial solutions are non-degenerate. We refer the reader
to \cite{Ramaswamy2013,Srikanth1993,ZhangLQ1992} for the precise
meaning of non-degenerate solutions.

In the general case $1<p<N$, among other results, Adimurthi and Yadava
\cite{AdimurthiYadava1994} proved the uniqueness of positive radial
solutions to the following prototype of equation (\ref{eq: Local Object})
\begin{equation}
\begin{cases}
-\Delta_{p}u=|u|^{p^{*}-2}u+\la|u|^{p-2}u & \text{in }B,\\
u=0 & \text{on }\pa B,
\end{cases}\label{eq: reduced equ.}
\end{equation}
where $\la\in\R$. The approach of Adimurthi and Yadava \cite{AdimurthiYadava1994},
roughly speaking, is as follows: Suppose that $u$ and $v$ are two
positive radial solutions to equation (\ref{eq: reduced equ.}). If
$u\ge v$ or $v\ge u$ in $B$, then it can be proved easily that
$u\equiv v$ in $B$. If $u\not\equiv v$ in $B$, then $u/v$ is
a positive continuous function on $\bar{B}$, the closure of $B$.
Then $0<\min_{\bar{B}}(u/v)<1$ and $0<\min_{\bar{B}}(v/u)<1$. They
excluded both cases by virtue of a generalized Pohozaev-type identity
from Ni and Serrin \cite{Ni-Serrin-1985-1,Ni-Serrin-1986-2} or Pucci
and Serrin \cite{Pucci-Serrin-1986}.

In the present paper, we follow the idea of Adimurthi and Yadava \cite{AdimurthiYadava1994}.
We obtain the following uniqueness result.
\begin{thm}
\label{thm: Uniqueness for local case } Assume that $1<p<N$, $-\wq<\mu<\bar{\mu}=((N-p)/p)^{p}$
and $0\le s<p$. If $\la\le0$, then equation (\ref{eq: Local Object})
admits no positive radial solution in $B$. If $\la>0$, then equation
(\ref{eq: Local Object}) admits at most one positive radial solution
in $B$.
\end{thm}
To follow the idea of Adimurthi and Yadava \cite{AdimurthiYadava1994},
first we establish a generalized Pohozaev-type identity for solutions
to equation (\ref{eq: Local Object}). This is done by combining the
generalized Pohozaev-type identity \cite{Ni-Serrin-1985-1,Ni-Serrin-1986-2,Pucci-Serrin-1986}
together with some apriori estimates on positive radial solutions
to equation (\ref{eq: Local Object}). Then we show that $u/v$ is
a positive continuous function on $\bar{B}$, if $u$ and $v$ are
two positive radial solutions to equation (\ref{eq: Local Object}).
This is done by a precise estimate on the asymptotic behavior of $u(x)$
and $v(x)$ as $x\to0$. Finally, we prove that $u\equiv v$ in the
same way as that of Adimurthi and Yadava \cite{AdimurthiYadava1994}.
Therefore the following estimates on the asymptotic behavior of positive
radial solutions to equation (\ref{eq: Local Object}) play a key
role in our argument.
\begin{thm}
\label{thm: Asym. beha. 1} Assume that $1<p<N$, $-\wq<\mu<\bar{\mu}=((N-p)/p)^{p}$,
$0\le s<p$ and $\la\in\R$. Let $u\in W_{0}^{1,p}(B)$ be a positive
radial solution to equation (\ref{eq: Local Object}). There exists
a constant $C>0$ such that
\begin{eqnarray*}
\lim_{|x|\to0}u(x)|x|^{\ga_{1}}=C & \text{ and } & \lim_{|x|\rightarrow0}|\nabla u(x)||x|^{\gamma_{1}+1}=|\ga_{1}|C.
\end{eqnarray*}
Here $\gamma_{1}$ is defined as in (\ref{eq: definition of ga-1 and ga-2}).
\end{thm}

\subsection{Classification of positive radial weak solutions to equation (\ref{eq: Global Object})}

Now we move to equation (\ref{eq: Global Object}). Note that equation
(\ref{eq: Global Object}) is invariant under the dilation
\begin{equation}
u^{\tau}(x)=\tau^{\frac{N-p}{p}}(\tau x)\label{eq: dilation}
\end{equation}
for $\tau>0$. That is, if $u$ is a solution to equation (\ref{eq: Global Object}),
then so is $u^{\tau}$. In the case $\mu=s=0$, equation (\ref{eq: Global Object})
is also invariant under translations. Taking into account the invariance
of equation (\ref{eq: Global Object}) with respect to (\ref{eq: dilation}),
we are concerned with the classification of positive radial solutions
(with respect to the origin) to equation (\ref{eq: Global Object})
in the Sobolev space ${\cal D}^{1,p}(\R^{N})$.

In many cases, exact forms of positive radial solutions to equation
(\ref{eq: Global Object}) in ${\cal D}^{1,p}(\R^{N})$ are known.
When $p=2$, equation (\ref{eq: Global Object}) is reduced to
\begin{eqnarray}
-\Delta u-\frac{\mu}{|x|^{2}}u=\frac{|u|{}^{2^{*}(s)-2}u}{|x|^{s}} &  & \text{in }\R^{N}.\label{eq: Catrina-Wang}
\end{eqnarray}
Assume that $-\wq<\mu<\bar{\mu}$ and $0\le s<2$. By Proposition
2.6 and Proposition 8.1 of Catrina and Wang \cite{Catrina-Wang2001},
every positive radial solution $u\in{\cal D}^{1,2}(\R^{N})$ to equation
(\ref{eq: Catrina-Wang}) are of the form
\[
u(x)=U_{2,\mu,s}^{\tau}(x)=\tau^{(N-2)/2}U_{2,\mu,s}(\tau x),
\]
for $\tau>0$, where
\[
U_{2,\mu,s}(x)=c_{2,\mu,s}\left(|x|^{\frac{2-s}{2}\left(1-\nu_{\mu}\right)}+|x|^{\frac{2-s}{2}\left(1+\nu_{\mu}\right)}\right)^{-\frac{N-2}{2-s}}
\]
with
\begin{eqnarray*}
\nu_{\mu}=\sqrt{1-\frac{\mu}{\bar{\mu}}} & \text{and} & c_{2,\mu,s}=\left(\frac{4(N-s)(\bar{\mu}-\mu)}{N-2}\right)^{\frac{N-2}{2(2-s)}}.
\end{eqnarray*}
For some special cases of equation (\ref{eq: Catrina-Wang}), the
explicit formula of $U_{2,\mu,s}$ was also obtained by many other
authors. We refer the reader to Aubin \cite{Aubin1976} and Talenti
\cite{Talent1976} for the case $\mu=s=0$, Lieb \cite{Lieb1983}
for the case $\mu=0$ and $0\le s<2$, Terracini \cite{Terracini}
for the case $0\le\mu<\bar{\mu}$ and $s=0$, and Chou-Chu \cite{ChouKS1993}
for the case $0\le\mu<\bar{\mu}$ and $0\le s<2$.

when $1<p<N$ and $\mu=0$, equation (\ref{eq: Global Object}) is
reduced to
\begin{eqnarray}
-\Delta_{p}u=\frac{|u|{}^{p^{*}(s)-2}u}{|x|^{s}} &  & \text{in }\R^{N}.\label{eq: Ghoussoub-Yuan}
\end{eqnarray}
Ghoussoub and Yuan \cite{Ghoussoub-Yuan2000} proved that all positive
radial solutions $u\in{\cal D}^{1,p}(\R^{N})$ to equation (\ref{eq: Ghoussoub-Yuan})
are of the form
\[
u(x)=U_{p,0,s}^{\tau}(x)=\tau^{(N-p)/p}U_{p,0,s}(\tau x),
\]
for $\tau>0$ (when $s=0$, it is also invariant with respect to translations),
where
\[
U_{p,0,s}(x)=c_{p,0,s}\left(1+|x|^{\frac{p-s}{p-1}}\right)^{-\frac{N-p}{p-s}},
\]
with
\[
c_{p,0,s}=\left((N-s)\left(\frac{N-p}{p-1}\right)^{p-1}\right)^{\frac{N-p}{p(p-s)}}.
\]
 In the case when $s=0$, above exact form was also obtained by Guedda
and V\'eron \cite{GueddaVeron1988}.

In the general case when $1<p<N$, $-\wq<\mu<\bar{\mu}$ $(\mu\ne0)$
and $0\le s<p$, the exact form for positive radial solutions to equation
(\ref{eq: Global Object}) in ${\cal D}^{1,p}(\R^{N})$ seems to be
unknown. In the particular case $1<p<N$, $0<\mu<\bar{\mu}$ and $s=0$,
that is, consider equation (\ref{eq: B-V-P-1}). Boumediene, Veronica
and Peral \cite{B-V-P2006} proved the uniqueness of positive radial
solutions to equation (\ref{eq: B-V-P-1}) in ${\cal D}^{1,p}(\R^{N})$,
up to a dilation (\ref{eq: dilation}). Moreover, they showed that
if $u\in{\cal D}^{1,p}(\R^{N})$ is a positive radial solution to
equation (\ref{eq: B-V-P-1}), then $u$ satisfies the estimates (\ref{est: optimal result of B-V-P-1})
and (\ref{est: optimal result of B-V-P-2}).

In this paper, we follow the argument of Boumediene, Veronica and
Peral \cite{B-V-P2006} and extend their uniqueness result to the
general case. We obtain the following result.
\begin{thm}
\label{thm: Uniqueness up to dilation}Assume that $1<p<N$, $-\wq<\mu<\bar{\mu}=((N-p)/p)^{p}$
and $0\le s<p$. Then up to a dilation (\ref{eq: dilation}), there
exists at most one positive radial solution $u\in{\cal D}^{1,p}(\R^{N})$
to equation (\ref{eq: Global Object}). Moreover, $u$ satisfies the
estimates (\ref{est: optimal result of B-V-P-1}) and (\ref{est: optimal result of B-V-P-2}).
\end{thm}
The paper is organized as follows. We prove Theorem \ref{thm: Uniqueness for local case }
in Section 2 and Theorem \ref{thm: Uniqueness up to dilation} in
Section 3. The proof of Theorem \ref{thm: Asym. beha. 1} is routine
and given in Section 4. Some preliminary results are given in the
Appendixes.

With no loss of generality, we assume throughout this paper that $B$
is the unit ball centered at the origin. By abuse of notation, we
write $u(x)=u(r)$ with $r=|x|$, whenever $u$ is a radial function.

\section{Proof of Theorem \ref{thm: Uniqueness for local case }}

In this section we prove Theorem \ref{thm: Uniqueness for local case }.
The following Pohozaev-type identity will be used in the proof of
Theorem \ref{thm: Uniqueness for local case }.
\begin{lem}
Let $u\in W_{0}^{1,p}(B)$ be a positive radial solution to equation
(\ref{eq: Local Object}). Then for any $0<r\le1$ we have
\begin{equation}
\begin{aligned}\la\int_{0}^{r}u(t)^{p}t^{N-1}dt & =\frac{p-1}{p}|u'(r)|^{p}r^{N}+\frac{N-p}{p}u(r)|u'(r)|^{p-2}u'(r)r^{N-1}\\
 & \quad+\frac{1}{p}\left(\mu r^{N-p}+\la r^{N}\right)u(r)^{p}+\frac{1}{p^{*}(s)}u(r)^{p^{*}(s)}r^{N-s}.
\end{aligned}
\label{eq: Pohozaev}
\end{equation}
\end{lem}
\begin{proof}
For any $0<a<r\le1$, we have from the Pohozaev-type variational identity
of Ni and Serrin \cite{Ni-Serrin-1985-1,Ni-Serrin-1986-2} or Pucci
and Serrin \cite{Pucci-Serrin-1986} that
\begin{eqnarray*}
\la\int_{a}^{b}u(t)^{p}t^{N-1}dt & = & \frac{p-1}{p}\left(|u'(r)|^{p}r^{N}-|u'(a)|^{p}a^{N}\right)+\frac{\mu}{p}\left(u(r)^{p}r^{N-p}-u(a)^{p}a^{N-p}\right)\\
 &  & \;+\frac{N-p}{p}\left(u(r)|u'(r)|^{p-2}u'(r)r^{N-1}-u(a)|u'(a)|^{p-2}u'(a)a^{N-1}\right)\\
 &  & \;+\left(\frac{1}{p^{*}(s)}\frac{u(r)^{p^{*}(s)}}{r^{s}}+\frac{\la}{p}u(r)^{p}\right)r^{N}-\left(\frac{1}{p^{*}(s)}\frac{u(a)^{p^{*}(s)}}{a^{s}}+\frac{\la}{p}u(a)^{p}\right)a^{N}.
\end{eqnarray*}
By Theorem \ref{thm: Asym. beha. 1} and the fact that $\ga_{1}<(N-p)/p$,
we obtain (\ref{eq: Pohozaev}) by sending $a\to0$ in above equality.
\end{proof}
We start the proof of Theorem \ref{thm: Uniqueness for local case }
with the following result.
\begin{prop}
\label{prop: Monotonicity of solutions} Assume that $\la>0$. If
$u,v\in W_{0}^{1,p}(B)$ are two positive radial solutions to equation
(\ref{eq: Local Object}) and $u\ge v$ in $B$, then $u\equiv v$
in $B$.\end{prop}
\begin{proof}
Suppose that $u,v\in W_{0}^{1,p}(B)$ are two positive radial solutions
to equation (\ref{eq: Local Object}) and $u\ge v$ in $B$. Set $\rho_{1}=u^{p^{*}(s)-p}r^{-s}+\la$
and $\rho_{2}=v^{p^{*}(s)-p}r^{-s}+\la$. Then $\rho_{1}\ge\rho_{2}>0$
in $B$ by assumptions. Applying Lemma \ref{lem: eigenvalue comparison},
we deduce that $u=v=0$ on the set $\{x\in B:\rho_{1}(x)>\rho_{2}(x)\}$.
Since $u,v$ are positive functions, we have that
\[
\{x\in B:\rho_{1}(x)>\rho_{2}(x)\}=\emptyset.
\]
That is, $\rho_{1}\equiv\rho_{2}$ in $B$. The proof of Proposition
\ref{prop: Monotonicity of solutions} is complete.
\end{proof}
We also need the following lemma.
\begin{lem}
\label{lem: Picard uniqueness } Assume that $1<p<N$. Let $v$ be
a positive solution to equation
\begin{equation}
\begin{cases}
-(r^{N-1}|v^{\prime}(r)|^{p-2}v^{\prime}(r))^{\prime}=g(r,v(r))r^{N-1} & \text{for }\frac{1}{2}<r<1,\\
v(1)=0,
\end{cases}\label{eq: Knaap-Peletier}
\end{equation}
where $g:[1/2,1]\times[0,\wq)\to\R$ is a function satisfying that
\begin{eqnarray}
|g(r,t)|\le C_{0}t^{p-1} &  & \text{for }(r,t)\in\left(\frac{1}{2},1\right)\times(0,\wq).\label{eq: growth condition of g}
\end{eqnarray}
Then
\[
v^{\prime}(1)<0.
\]
\end{lem}
\begin{proof}
This lemma should be well known. But as we did not find a proper reference,
we give a proof here for completeness. Since $v$ is a positive solution
and $v(1)=0$, then $v^{\prime}(1)\le0.$ Suppose that Lemma \ref{lem: Picard uniqueness }
is not true. That is, we suppose that
\begin{equation}
v^{\prime}(1)=0.\label{eq: star-1-1}
\end{equation}
Integrate each side of equation (\ref{eq: Knaap-Peletier}) from $r$
to $1$. We obtain, by (\ref{eq: star-1-1}), that
\[
r^{N-1}|v^{\prime}(r)|^{p-2}v^{\prime}(r)=\int_{r}^{1}g(s,u(s))s^{N-1}ds,
\]
for all $1/2\le r<1$. It follows from above equality and (\ref{eq: growth condition of g})
that
\[
|v^{\prime}(r)|^{p}\le C(1-r)^{\frac{1}{p-1}}\int_{r}^{1}v^{p-1}d\tau,
\]
for all $1/2\le r<1$. Combine H\"older's inequality and the assumption
that $v(1)=0$. We obtain that
\begin{eqnarray*}
v(r)^{p} & = & p\int_{r}^{1}v(\tau)^{p-1}v^{\prime}(\tau)d\tau\\
 & \le & C\left(\int_{r}^{1}v(\tau)^{p}d\tau\right)^{\frac{p-1}{p}}\left(\int_{r}^{1}|v^{\prime}(\tau)|^{p}d\tau\right)^{\frac{1}{p}}\\
 & \le & C(1-r)^{\frac{1}{p-1}}\int_{r}^{1}v(\tau)^{p}d\tau,
\end{eqnarray*}
for all $1/2\le r<1$. Define $w(s)=v(1-s)$ for $0\le s\le1/2$.
Above equality is equivalent to
\[
w(s)^{p}\le Cs^{\frac{1}{p-1}}\int_{0}^{s}w(\tau)^{p}d\tau
\]
for all $0\le s\le1/2$. Note that $w(0)=0$. It follows from the
Gronwall's inequality that
\begin{eqnarray*}
w(s)=0 &  & \text{for all }0\le s\le\frac{1}{2},
\end{eqnarray*}
which is equivalent to
\begin{eqnarray*}
v(r)=0 &  & \text{for all }\frac{1}{2}\le r\le1.
\end{eqnarray*}
We reach a contradiction, as we assume that $v$ is positive. The
proof of Lemma \ref{lem: Picard uniqueness } is complete.
\end{proof}
Now we prove Theorem \ref{thm: Uniqueness for local case }.

\begin{proof}[Proof of Theorem \ref{thm: Uniqueness for local case }]
We have two cases:

\emph{Case 1: $\la\le0$; }

\emph{Case 2: $\la>0$. }

Consider Case 1. Suppose that $u\in W_{0}^{1,p}(B)$ is a positive
radial solution to equation (\ref{eq: Local Object}). Then $u$ satisfies
the identity (\ref{eq: Pohozaev}). Take $r=1$ in (\ref{eq: Pohozaev}).
We obtain that
\begin{equation}
0\ge\la\int_{a}^{b}u(t)^{p}t^{N-1}dt=\frac{p-1}{p}|u'(1)|^{p}\ge0.\label{eq: case 1}
\end{equation}
If $\la<0$, then $u\equiv0$ in $B$ by (\ref{eq: case 1}). We obtain
a contradiction. If $\la=0$, then $u^{\prime}(1)=0$ by (\ref{eq: case 1}).
Note that $u(1)=0$. Apply Lemma \ref{lem: Picard uniqueness } to
$u$ with $g(r,u)=\left(\mu r^{-p}+u^{p^{*}(s)-p}r^{-s}+\la\right)u^{p-1}$.
We obtain that $u(r)\equiv0$ for $1/2\le r\le1$. We reach a contradiction.
Hence in Case 1, that is, $\la\le0$, there has no positive radial
solution to equation (\ref{eq: Local Object}) in $B$. We remark
that when $\mu=0$, Adimurthi and Yanava \cite{AdimurthiYadava1994}
pointed out that an observation of Knaap and Peletier \cite{Knaap-Peletier1989}
implies that $u(r)\equiv0$ for $0<r\le1$. They also pointed out
that a more general theorem given by Franchi, Lanconelli and Serrin
\cite{F-L-Serrin-2009} also claims that $u(r)\equiv0$ for $0<r\le1$.

Consider Case 2. Suppose that $u,v\in W_{0}^{1,p}(B)$ are two positive
radial solutions to equation (\ref{eq: Local Object}). We prove that
$u\equiv v$ in $B$. Suppose, on the contrary, that $u\not\equiv v$
in $B$. That is,
\begin{eqnarray}
u(r)\not\equiv v(r) &  & \text{for }0<r<1.\label{eq: contradiction assumption}
\end{eqnarray}
Let
\begin{eqnarray*}
w(r)=\frac{u(r)}{v(r)} &  & \text{for }0<r<1.
\end{eqnarray*}
Then $w$ is a positive continuous function in $(0,1)$.

First, we claim that $w$ can be extended to $r=0$ and $r=1$ such
that $w$ is a positive continuous function on $[0,1]$. Indeed, by
Theorem \ref{thm: Asym. beha. 1}, there exist constants $C_{u},C_{v}>0$
such that
\begin{eqnarray*}
\lim_{r\to0}u(r)r^{\ga_{1}}=C_{u} & \text{ and } & \lim_{r\to0}v(r)r^{\ga_{1}}=C_{v}.
\end{eqnarray*}
 Then we have that
\[
\lim_{r\to0}w(r)=\frac{C_{u}}{C_{v}}>0.
\]
Thus we can extend $w$ continuously to $r=0$ by setting $w(0)=C_{u}/C_{v}$.
On the other hand, by L'Hospital's rule, we have that
\[
\lim_{r\to1}w(r)=\frac{u^{\prime}(1)}{v^{\prime}(1)}>0,
\]
since both $u^{\prime}(1)$ and $v^{\prime}(1)$ are negative by Lemma
\ref{lem: Picard uniqueness }. Hence we can extend $w$ continuously
to $r=1$ by setting $w(1)=u^{\prime}(1)/v^{\prime}(1)$. Then $w$
is a positive continuous function on $[0,1]$.

Next, set
\[
\al=\inf_{r\in[0,1]}w(r).
\]
Then $\al>0$. We claim that $\al<1$. Otherwise, if $\al\ge1$, then
$u\ge v$ in $(0,1)$. Proposition \ref{prop: Monotonicity of solutions}
implies that $u\equiv v$ in $(0,1)$. This contradicts to (\ref{eq: contradiction assumption}).
Hence $0<\al<1.$ Since $w$ is continuous on $[0,1]$, $\al$ can
be achieved by $w$ on $[0,1]$. Let $r_{\al}$ be such that
\[
r_{\al}=\inf\{t\in[0,1]:w(t)=\al\}.
\]
We claim that
\begin{equation}
r_{\al}=0.\label{eq: r-alpha=00003D0}
\end{equation}
Otherwise, we have $0<r_{\al}\le1$. If $r_{\al}=1$, that is, $w(1)=\al$
and $w(r)>\al$ for $0\le r<1$. Then we deduce that $u'(1)=\al v'(1)$,
and $u(r)>\al v(r)$ for $0\le r<1$. Take $r=r_{\al}=1$ in (\ref{eq: Pohozaev}).
Since both $u,v$ satisfy (\ref{eq: Pohozaev}), we obtain that
\[
0<\la\int_{0}^{1}\left(u(t)^{p}-\al^{p}v(t)^{p}\right)t^{N-1}dt=\frac{p-1}{p}\left(|u'(1)|^{p}-\al^{p}|v'(1)|^{p}\right)=0.
\]
We reach a contradiction. If $0<r_{\al}<1$, then $w(r_{\al})=\al$
and $w(r)>\al$ for $0\le r<r_{\al}$. Note that $w^{\prime}(r_{\al})=0$.
We deduce that $u>\al v$ in $(0,r_{\al})$, $u(r_{\al})=\al v(r_{\al})$
and $u^{\prime}(r_{\al})=\al v^{\prime}(r_{\al})$. Take $r=r_{\al}$
in (\ref{eq: Pohozaev}). We obtain that
\begin{eqnarray*}
0<\la\int_{0}^{r_{\al}}\left(u(t)^{p}-\al^{p}v(t)^{p}\right)t^{N-1}dt & = & \frac{r_{\al}^{N-s}}{p^{*}(s)}\left(u(r_{\al})^{p^{*}(s)}-\al^{p}v(r_{\al})^{p^{*}(s)}\right)\\
 & = & \frac{r_{\al}^{N-s}}{p^{*}(s)}\left(\al^{p^{*}(s)}-\al^{p}\right)v(r_{\al})^{p^{*}(s)}\\
 & < & 0,
\end{eqnarray*}
since $0<\al<1$ and $p^{*}(s)>p$. We reach a contradiction. This
proves (\ref{eq: r-alpha=00003D0}).

Therefore we obtain that $w(0)=\al<1.$ Recall that $w(0)=C_{u}/C_{v}$.
Hence
\[
C_{u}<C_{v}.
\]
Similarly, consider $\tilde{w}(r)=v(r)/u(r)$. Repeat above procedure
with respect to $\tilde{w}(r)$. We obtain that $C_{v}/C_{u}=\tilde{w}(0)<1$.
Hence
\[
C_{v}<C_{u}.
\]
We reach a contradiction. Therefore $u\equiv v$ in $(0,1)$. The
proof of Theorem \ref{thm: Uniqueness for local case } is complete.
\end{proof}

\section{Proof of Theorem \ref{thm: Uniqueness up to dilation} }

In this section we prove Theorem \ref{thm: Uniqueness up to dilation}.
Before giving the proof of Theorem \ref{thm: Uniqueness up to dilation},
let us revisit the following prototype of equation (\ref{eq: Global Object})
\begin{eqnarray}
-\De u=|u|^{2^{*}-2}u &  & \text{in }\R^{N},\label{eq: Caffarelli-G-S}
\end{eqnarray}
where $2^{*}=2N/(N-2)$ and $N\ge3$. Let $u\in{\cal D}^{1,2}(\R^{N})$
be a positive radial solution to equation (\ref{eq: Caffarelli-G-S}).
Applying the transform
\begin{eqnarray}
t=\log r & \text{and} & \psi(t)=r^{\frac{N-2}{2}}u(r),\label{eq: E-F transform-1}
\end{eqnarray}
for $r\in(0,\wq)$, we deduce that
\begin{eqnarray}
\psi^{\prime\prime}(t)-\left(\frac{N-2}{2}\right)^{2}\psi(t)+\psi^{2^{*}-1}=0 &  & \text{in }\R.\label{eq: Caffarelli-G-S-2}
\end{eqnarray}
Solving equation (\ref{eq: Caffarelli-G-S-2}) (see details in e.g.
\cite[Section 1]{Caffarelli1989}) and taking into account that $u\in{\cal D}^{1,2}(\R^{N})$,
we obtain that
\begin{eqnarray*}
u(x)=\left(\frac{\la\sqrt{N(N-2)}}{\la^{2}+|x-x_{0}|^{2}}\right)^{\frac{N-2}{2}}, &  & \la>0,x_{0}\in\R^{N}.
\end{eqnarray*}
So this gives the exact form of $u$.

In above approach the transform (\ref{eq: E-F transform-1}) turns
equation (\ref{eq: Caffarelli-G-S}) into ordinary differential equation
(\ref{eq: Caffarelli-G-S-2}) which can be solved explicitly. In the
general case $1<p<N$, $-\wq<\mu<\bar{\mu}$ and $0\le s<p$, a similar
type of transform to (\ref{eq: E-F transform-1}) will be used to
turn equation (\ref{eq: Global Object}) into an ordinary differential
equation system. Then we follow the argument of Boumediene, Veronica
and Peral \cite{B-V-P2006} to establish the uniqueness (up to a dilation)
of positive radial weak solutions to equation (\ref{eq: Global Object}).

Let $u\in{\cal D}^{1,p}(\R^{N})$ be a positive radial solution to
equation (\ref{eq: Global Object}). Then we have
\begin{equation}
\int_{0}^{\wq}\left({\displaystyle \frac{|u(r)|^{p^{*}(s)}}{r^{s}}}+|u^{\prime}(r)|^{p}\right)r^{N-1}dr=\frac{1}{\om_{N-1}}\int_{\R^{N}}\left({\displaystyle \frac{|u|^{p^{*}(s)}}{|x|^{s}}}+|\na u|^{p}\right)dx<\wq,\label{eq: energy}
\end{equation}
where $\om_{N-1}$ is the surface measure of the unit sphere in $\R^{N}$.
And $u$ is a solution to equation
\begin{equation}
\begin{cases}
-\left(r^{N-1}|u'(r)|^{p-2}u'(r)\right)^{\prime}=\left({\displaystyle \frac{\mu}{r^{p}}+\frac{u(r){}^{p^{*}(s)-p}}{r^{s}}}\right)u(r)^{p-1}r^{N-1}, & r\in(0,\wq),\\
u(r)>0, & r\in(0,\wq).
\end{cases}\label{eq: ODE global}
\end{equation}
Apply the transform:
\begin{eqnarray}
t=\log r, & y(t)=r^{\de}u(r), & z(t)=r^{(p-1)(\de+1)}|u^{\prime}(r)|^{p-2}u^{\prime}(r),\label{eq: E-F transform}
\end{eqnarray}
where we denote $\de=(N-p)/p$ in this section. We obtain by equation
(\ref{eq: ODE global}) that $y$ satisfies
\begin{eqnarray}
y^{\prime}=\de y+|z|^{\frac{1}{p-1}-1}z, & y>0 & \text{ in }\R,\label{eq: ODE of y}
\end{eqnarray}
and $z$ satisfies
\begin{eqnarray}
z^{\prime}=-\de z-y{}^{p^{*}(s)-1}-\mu y^{p-1} &  & \text{ in }\R.\label{eq: ODE of z}
\end{eqnarray}
Define $V:\R^{2}\to\R$ by
\begin{equation}
V(a,b)=\frac{1}{p^{*}(s)}|a|^{p^{*}(s)}+\frac{\mu}{p}|a|^{p}+\de ab+\frac{1}{p^{\prime}}|b|^{p^{\prime}}.\label{eq:V}
\end{equation}
Here $p^{\prime}=p/(p-1)$. It follows from equations (\ref{eq: ODE of y})
and (\ref{eq: ODE of z}) that
\begin{eqnarray*}
\frac{d}{dt}\left(V(y(t),z(t))\right)=0, &  & \forall t\in\R.
\end{eqnarray*}
Hence there is a constant $K$ such that
\begin{eqnarray}
V(y(t),z(t))\equiv K, &  & \forall t\in\R.\label{eq: first complete integral}
\end{eqnarray}
 Since $u\in{\cal D}^{1,p}(\R^{N})$ is a radial function, we have
(see \cite[Corollary II.1]{Lions1982} and its proof)
\[
\lim_{r\to0}r^{\de}u(r)=\lim_{r\to\wq}r^{\de}u(r)=0.
\]
Thus
\[
\lim_{|t|\to\wq}y(t)=0.
\]
Note also  that by (\ref{eq: energy}) we have
\[
\liminf_{r\to0}r^{\de+1}|u^{\prime}(r)|=\liminf_{r\to\wq}r^{\de+1}|u^{\prime}(r)|=0.
\]
Hence
\[
\liminf_{|t|\rightarrow\infty}|z(t)|=0.
\]
Sending $|t|\to\wq$ in (\ref{eq: first complete integral}), we deduce
that $K=0$, that is,
\begin{eqnarray}
\frac{1}{p^{*}(s)}y(t)^{p^{*}(s)}+\frac{\mu}{p}y(t){}^{p}+\de y(t)z(t)+\frac{1}{p^{\prime}}|z(t)|^{p^{\prime}}=0, &  & \forall t\in\R.\label{eq:K is zero}
\end{eqnarray}

We claim that $y$ is bounded on $\R$. Precisely, set
\begin{equation}
M=\left(\frac{p^{*}(s)(\bar{\mu}-\mu)}{p}\right)^{\frac{1}{p^{*}(s)-p}}.\label{eq: M}
\end{equation}

\begin{lem}
\label{lem: 4.1} We have
\begin{eqnarray}
y(t)\le M, &  & \forall t\in\R.\label{eq: boundedness of y}
\end{eqnarray}
Moreover, $y(t_{0})=M$ at a point $t_{0}\in\R$ if and only if $\de y(t_{0})=-|z(t_{0})|^{\frac{1}{p-1}-1}z(t_{0})$. \end{lem}
\begin{proof}
Recall that Young's inequality gives that
\begin{eqnarray*}
ab\leq\frac{1}{p}|a|^{p}+\frac{1}{p^{\prime}}|b|^{p^{\prime}}, &  & \forall\, a,b\in\R,
\end{eqnarray*}
and the equality holds if and only if $|a|=|b|^{\frac{1}{p-1}}$ and
$ab\ge0$. Hence
\begin{eqnarray}
-\de y(t)z(t)\le\frac{\de^{p}}{p}y(t)^{p}+\frac{1}{p^{\prime}}|z(t)|^{p^{\prime}}, &  & \forall t\in\R,\label{eq: 4.star-1}
\end{eqnarray}
and the equality holds at some $t=t_{0}$ if and only if $\de y(t_{0})=|z(t_{0})|^{\frac{1}{p-1}}$
and $z(t_{0})<0$. Note that $\de^{p}=\bar{\mu}=\left((N-p)/p\right)^{p}$.
Combining (\ref{eq: 4.star-1}) and (\ref{eq:K is zero}) gives us
that
\[
\frac{1}{p^{*}(s)}y(t){}^{p^{*}(s)}\le\frac{\bar{\mu}-\mu}{p}y(t){}^{p},
\]
which implies (\ref{eq: boundedness of y}), and the equality holds
at $t=t_{0}\in\R$ if and only if $\de y(t_{0})=-|z(t_{0})|^{\frac{1}{p-1}-1}z(t_{0})$.
This proves the lemma.
\end{proof}
Since $y$ is continuous in $\R$ and $y(t)\to0$ as $|t|\to\wq$,
$y$ achieves its maximum in $\R$. Let $t_{0}\in\R$ be such that
$y(t_{0})=\max_{\R}y$. Then $t_{0}$ is a critical point of $y$,
that is, $y^{\prime}(t_{0})=0$. By equation (\ref{eq: ODE of y}),
we obtain that $\de y(t_{0})=-|z(t_{0})|^{\frac{1}{p-1}-1}z(t_{0}).$
Then Lemma \ref{lem: 4.1} implies that $y(t_{0})=M$. We claim that
$t_{0}$ is the unique critical point of $y$ in $\R$. Indeed, suppose
that $t_{1}\in\R$ is another critical point of $y$. Then combining
equation (\ref{eq: ODE of y}) and Lemma \ref{lem: 4.1} yields that
$y(t_{1})=M$. With no loss of generality, we assume that $t_{1}<t_{0}$.
We prove that $y\equiv M$ in $[t_{1},t_{0}]$. Otherwise, there exists
$t_{2}\in(t_{1},t_{0})$ such that $y(t_{2})=\min_{[t_{1},t_{0}]}y<M$.
Then $y^{\prime}(t_{2})=0$. Combining equation (\ref{eq: ODE of y})
and Lemma \ref{lem: 4.1} again yields that $y(t_{2})=M$. We reach
a contradiction. Hence $y\equiv M$ on $[t_{1},t_{0}]$. But then
we have $y^{\prime}\equiv0$ on $[t_{1},t_{0}]$. Consider equation
(\ref{eq: ODE of y}) on the interval $[t_{1},t_{0}]$. We obtain
that $z\equiv-\left(\de M\right)^{p-1}$ on $[t_{1},t_{0}]$. Then
we derive from equation (\ref{eq: ODE of z}) that
\[
\de\left(\de M\right)^{p-1}-M^{p^{*}(s)-1}-\mu M^{p-1}=0,
\]
which implies that $M=\left(\bar{\mu}-\mu\right)^{\frac{1}{p^{*}(s)-p}}$.
We reach a contradiction to (\ref{eq: M}). Hence $t_{0}$ is the
unique critical point of $y$ in $\R$. Thus $y^{\prime}(t)>0$ for
$t<t_{0}$ and $y^{\prime}(t)<0$ for $t>t_{0}$. Note that both equations
(\ref{eq: ODE of y}) and (\ref{eq: ODE of z}) are invariant under
translations. Therefore, up to a translation, we assume in the rest
of this section that $y$ satisfies
\begin{equation}
\begin{cases}
y(0)=\max_{t\in R}y=M,\text{ and }\\
y^{\prime}>0\text{ in }(-\infty,0)\text{ and }y^{\prime}<0\text{ in }(0,\infty).
\end{cases}\label{eq: translation of y}
\end{equation}
It follows immediately from equation (\ref{eq: ODE of y}) and (\ref{eq: translation of y})
that
\begin{equation}
z(0)=-\left(\de M\right)^{p-1}.\label{eq: z(0)}
\end{equation}

\begin{lem}
\label{lem: properties of z} For the function $z$, we have,

(1) $z$ is a bounded continuous function on $\R$;

(2) in the case $0\le\mu<\bar{\mu}$, $z(t)<0$ for all $t\in\R$;

(3) in the case $\mu<0$, there exists a unique point $t_{-}\in\R$,
$t_{-}<0$, such that $z>0$ in $(-\infty,t_{-})$ and $z<0$ in $(t_{-},\infty)$. \end{lem}
\begin{proof}
(1) The boundedness of $z$ follows from (\ref{eq:K is zero}) and
boundedness of $y$.

(2) In the case $0\le\mu<\bar{\mu}$, it follows from (\ref{eq:K is zero})
easily that $z(t)<0$ for all $t\in\R$.

(3) Consider the case $\mu<0$. We claim that there exists a constant
$L>0$ sufficiently large such that $z(t)>0$ for $t<-L$. Indeed,
since $z$ is bounded by (1), we have $e^{\de t}z(t)\to0$ as $t\to-\wq$.
Solve equation (\ref{eq: ODE of z}). We deduce that
\begin{eqnarray*}
e^{\de t}z(t)=\int_{-\wq}^{t}e^{\de s}y(s)^{p-1}\left(-\mu-y(s)^{p^{*}(s)-p}\right)ds &  & \text{for }t\in\R.
\end{eqnarray*}
Since $\mu<0$ and $y(t)\to0$ as $t\to-\wq$, there exists $L>0$
sufficiently large such that $-\mu-y(s)^{p^{*}(s)-p}>0$ for $s<-L$.
Thus $e^{\de t}z(t)>0$ for $t<-L$. This proves the claim.

Note that by (\ref{eq: z(0)}) we have $z(0)<0$. Hence, by above
claim, the set
\[
Z=\{t\in\R:z(t)=0\}
\]
is not empty. To prove (3), it is enough to prove that $Z$ consists
of only one point. Let $t_{0}\in\R$ be an arbitrary point in $Z$.
Then $z(t_{0})=0$. We show that $t_{0}$ can be uniquely determined.
Substitute $t=t_{0}$ into equation (\ref{eq: ODE of y}). We obtain
that $y^{\prime}(t_{0})=\de y(t_{0})>0$. Hence $t_{0}\in(-\wq,0)$
by (\ref{eq: translation of y}). Substitute $t=t_{0}$ into equation
(\ref{eq:K is zero}). We obtain that
\begin{equation}
y(t_{0})=\left(-p^{*}(s)\mu/p\right)^{1/(p^{*}(s)-p)}.\label{eq: star-1}
\end{equation}
Since $y$ is strictly monotone in $(-\wq,0)$ by (\ref{eq: translation of y}),
we find that $t_{0}$ is the unique point in $(-\wq,0)$ which satisfies
(\ref{eq: star-1}). This proves that $Z$ consists of only one point.
Denote by $t_{-}$ the point in $Z$. The proof of (3) is complete.
\end{proof}
Now we study the asymptotic behaviors of $y$ and $z$. Let $\ga\in\R$
be an arbitrary number and define
\begin{eqnarray*}
y_{\gamma}(t)=e^{(\gamma-\delta)t}y(t), &  & t\in\R.
\end{eqnarray*}
 By (\ref{eq: ODE of y}) and (\ref{eq: translation of y}), we have
\begin{equation}
\begin{cases}
y_{\gamma}^{\prime}(t)=\left(\ga-H(t)\right)y_{\gamma}(t), & t\in\R,\\
y_{\gamma}(0)=M,
\end{cases}\label{eq:y-gamma}
\end{equation}
where $M$ is defined as in (\ref{eq: M}) and $H:\R\to\R$ is defined
by
\begin{eqnarray}
H(t)=-\frac{|z(t)|^{\frac{1}{p-1}-1}z(t)}{y(t)}, &  & t\in\R.\label{eq: definition of H}
\end{eqnarray}
Note that $H$ is a continuous function on $\R$. Let $t_{-}$ be
the number defined as in Lemma \ref{lem: properties of z} in the
case $\mu<0$. $H$ is continuously differentiable on $\R$ except
at the point $t=t_{-}$ in the case $\mu<0$.

The function $H$ plays a key role in the proof of Theorem \ref{thm: Uniqueness up to dilation}.
We derive the equation satisfied by $H$. For $t\ne t_{-}$, We have
that
\begin{eqnarray*}
H^{\prime} & = & y^{-2}\left(|z|^{\frac{1}{p-1}-1}zy^{\prime}-\frac{1}{p-1}|z|^{\frac{1}{p-1}-1}z^{\prime}y\right)\\
 & = & y^{-2}|z|^{\frac{2-p}{p-1}}\left((\delta y+|z|^{\frac{1}{p-1}-1}z)z-\frac{1}{p-1}y(-\delta z-y{}^{p^{*}(s)-1}-\mu y{}^{p-1})\right)\\
 & = & y^{-2}|z|^{\frac{2-p}{p-1}}\left(p^{\prime}\delta yz+|z|^{p^{\prime}}+\frac{1}{p-1}y^{p^{*}(s)}+\frac{\mu}{p-1}y^{p}\right)\\
 & = & p^{\prime}y^{-2}|z|^{\frac{2-p}{p-1}}\left(\delta yz+\frac{1}{p^{\prime}}|z|^{p^{\prime}}+\frac{1}{p}y^{p^{*}(s)}+\frac{\mu}{p}y^{p}\right)\\
 & = & \frac{p^{*}(s)-p}{p^{*}(s)(p-1)}y^{p^{*}(s)-2}|z|^{\frac{2-p}{p-1}},
\end{eqnarray*}
where the second equality follows from equations (\ref{eq: ODE of y})
and (\ref{eq: ODE of z}), and the last equality follows from (\ref{eq:K is zero}).
Thus by the definition (\ref{eq: definition of H}) of $H$, we obtain
that
\begin{eqnarray}
H^{\prime}(t)=\frac{p^{*}(s)-p}{p^{*}(s)(p-1)}y(t)^{p^{*}(s)-p}|H(t)|^{2-p} &  & \text{for }t\ne t_{-}.\label{eq:Equ. of H-1}
\end{eqnarray}
We remark that equation (\ref{eq:Equ. of H-1}) holds at $t=t_{-}$
if $0\le\mu$. On the other hand, by (\ref{eq:K is zero}) we have
that
\begin{eqnarray*}
\frac{1}{p^{*}(s)}y{}^{p^{*}(s)-p}+\frac{1}{p^{\prime}}\frac{|z|^{p^{\prime}}}{y^{p}}+\delta\frac{z}{y^{p-1}}+\frac{\mu}{p}\equiv0 &  & \text{in }\R.
\end{eqnarray*}
Recall that $\Ga_{\mu}$ is defined as in (\ref{eq: definition of Gamma-mu}).
We obtain that
\begin{eqnarray}
\frac{p}{p^{*}(s)}y(t){}^{p^{*}(s)-p}=-\Ga_{\mu}(H(t)) &  & \text{in }\R.\label{eq:y and H}
\end{eqnarray}
Combining equation (\ref{eq:Equ. of H-1}) and equation (\ref{eq:y and H})
yields that
\begin{eqnarray}
H^{\prime}(t)=-\frac{p^{*}(s)-p}{p(p-1)}|H(t)|^{2-p}\Ga_{\mu}(H(t)) &  & \text{for }t\ne t_{-}.\label{eq: equ. of H-2}
\end{eqnarray}
That is, $H$ satisfies equation (\ref{eq: equ. of H-2}). We remark
that when $0\le\mu<\bar{\mu}$, (\ref{eq: equ. of H-2}) holds for
all $t\in\R$.

We claim that
\begin{eqnarray}
\lim_{t\to-\wq}H(t)=\ga_{1} & \text{ and } & \lim_{t\to\wq}H(t)=\ga_{2}.\label{eq: limit of H}
\end{eqnarray}
Indeed, Let $t\to-\wq$ and $t\to\wq$ in equation (\ref{eq:y and H})
respectively. we obtain that
\begin{equation}
\lim_{t\to-\wq}\Ga_{\mu}(H(t))=\lim_{t\to\wq}\Ga_{\mu}(H(t))=0.\label{eq: 4.star-2}
\end{equation}
By equation (\ref{eq:Equ. of H-1}), $H$ is strictly increasing in
$\R$. Hence there exist $a,b$, $-\wq\le a<b\le\wq$, such that $\lim_{t\to-\wq}H(t)=a$
and $\lim_{t\to\wq}H(t)=b$. Note that $\Ga_{\mu}(\ga)\to\wq$ as
$|\ga|\to\wq$. Hence (\ref{eq: 4.star-2}) implies that $a,b$ are
finite. Then $\Ga_{\mu}(a)=\Ga_{\mu}(b)=0$. Since $\ga_{1}$ and
$\ga_{2}$ are the only two roots of $\Ga_{\mu}$ in $\R$ and $\ga_{1}<\ga_{2}$,
we obtain that $a=\ga_{1}$ and $b=\ga_{2}$, and then the claim is
proved. Therefore, the monotonicity of $H$ implies that
\begin{eqnarray}
\ga_{1}<H(t)<\ga_{2}, &  & \forall\: t\in\R.\label{eq: range of H}
\end{eqnarray}
We claim that
\begin{equation}
\int_{-\infty}^{0}(H(s)-\gamma_{1})ds+\int_{0}^{\infty}(\gamma_{2}-H(s))ds<\infty.\label{eq:integrable of H}
\end{equation}
To prove (\ref{eq:integrable of H}), rewrite $\Ga_{\mu}$ by $\Ga_{\mu}(s)=(s-\gamma_{1})(s-\gamma_{2})\tilde{\Ga}_{\mu}(s)$,
where $\tilde{\Ga}_{\mu}$ is a continuous function on $\R$ satisfying
$\inf_{\R}\tilde{\Ga}_{\mu}>0$. Then by change of variable, we have
that
\begin{eqnarray*}
\int_{-\infty}^{2t_{-}}(H(s)-\gamma_{1})ds & = & \int_{\gamma_{1}}^{H(2t_{-})}\frac{(\tau-\gamma_{1})d\tau}{\tau^{\prime}(s)}\\
 & = & \int_{\gamma_{1}}^{H(2t_{-})}\frac{(\tau-\gamma_{1})d\tau}{-\frac{p^{*}(s)-p}{p(p-1)}|\tau|^{2-p}\Ga_{\mu}(\tau)}\\
 & = & \int_{\gamma_{1}}^{H(2t_{-})}\frac{p(p-1)d\tau}{(p^{*}(s)-p)|\tau|^{2-p}(\gamma_{2}-\tau)\tilde{\Ga}_{\mu}(\tau)}\\
 & \le & C\int_{\gamma_{1}}^{H(2t_{-})}|\tau|^{p-2}d\tau\\
 & < & \wq,
\end{eqnarray*}
where $C=p(p-1)/\left((p^{*}(s)-p)(\ga_{2}-H(2t_{-}))\inf_{\R}\tilde{\Ga}_{\mu}\right)$.
Similarly, we have that

\begin{eqnarray*}
\int_{0}^{\infty}(\gamma_{2}-H(s))ds & = & \int_{\delta}^{\gamma_{2}}\frac{(\gamma_{2}-\tau)}{\tau^{\prime}(s)}d\tau\\
 & = & \int_{\delta}^{\gamma_{2}}\frac{(\gamma_{2}-\tau)d\tau}{-\frac{p^{*}(s)-p}{p(p-1)}|\tau|^{2-p}\Ga_{\mu}(\tau)}\\
 & = & \int_{\delta}^{\gamma_{2}}\frac{p(p-1)d\tau}{(p^{*}(s)-p)|\tau|^{2-p}(\tau-\gamma_{1})\tilde{\Ga}_{\mu}(\tau)}\\
 & \le & C^{\prime}\int_{\delta}^{\gamma_{2}}\tau{}^{p-2}d\tau\\
 & < & \wq,
\end{eqnarray*}
where $C^{\prime}=p(p-1)/\left((p^{*}(s)-p)(\de-\ga_{1})\inf_{\R}\tilde{\Ga}_{\mu}\right)$.
This proves (\ref{eq:integrable of H}).

Now we are in a position to prove Theorem \ref{thm: Uniqueness up to dilation}.

\begin{proof}[Proof of Theorem  \ref{thm: Uniqueness up to dilation}]
Let $u\in{\cal D}^{1,p}(\R^{N})$ be a positive radial solution to
equation (\ref{eq: Global Object}) and $(y,z)$ defined by the transform
(\ref{eq: E-F transform}) with respect to $u$. Let $H$ be defined
as in (\ref{eq: definition of H}). First we show that $u$ satisfies
(\ref{est: optimal result of B-V-P-1}) and (\ref{est: optimal result of B-V-P-2}).

Integrate (\ref{eq:y-gamma}). We obtain that
\begin{eqnarray}
e^{(\gamma-\delta)t}y(t)=M\exp\left(\int_{0}^{t}(\gamma-H(\tau))d\tau\right) &  & \text{for }t\in\R.\label{eq:integral of y gamma}
\end{eqnarray}
Hence we derive that
\begin{equation}
\begin{cases}
{\displaystyle \lim_{t\rightarrow-\infty}}e^{(\gamma_{1}-\delta)t}y(t)=M\exp\left(\int_{-\infty}^{0}(H-\gamma_{1})d\tau\right)=:C_{1}, & \text{and}\\
{\displaystyle \lim_{t\rightarrow\infty}}e^{(\gamma_{2}-\delta)t}y(t)=M\exp\left(\int_{0}^{\infty}(\gamma_{2}-H)d\tau\right)=:C_{2},
\end{cases}\label{eq:limit of y}
\end{equation}
which is equivalent to (\ref{est: optimal result of B-V-P-1}). Since
\begin{equation}
|z(t)|^{\frac{1}{p-1}-1}z(t)=H(t)y(t),\label{eq: z determined by y and H}
\end{equation}
we derive from (\ref{eq: limit of H}) and (\ref{eq:limit of y})
that
\begin{equation}
\begin{cases}
{\displaystyle \lim_{t\rightarrow-\infty}}e^{(\gamma_{1}-\delta)t}|z(t)|^{\frac{1}{p-1}}=C_{1}|\gamma_{1}|, & \text{and}\\
{\displaystyle \lim_{t\rightarrow\infty}}e^{(\gamma_{2}-\delta)t}|z(t)|^{\frac{1}{p-1}}=C_{2}\gamma_{2},
\end{cases}\label{eq:limit of z}
\end{equation}
which is equivalent to (\ref{est: optimal result of B-V-P-2}). This
proves that $u$ satisfies (\ref{est: optimal result of B-V-P-1})
and (\ref{est: optimal result of B-V-P-2}).

Next we prove the uniqueness of $u$ up to a dilation. Suppose that
$u_{1},u_{2}\in{\cal D}^{1,p}(\R^{N})$ are two positive radial solutions
to equation (\ref{eq: Global Object}). Define $(y_{i},z_{i})$ by
the transform (\ref{eq: E-F transform}) with respect to $u_{i}$
for $i=1,2$. Define $H_{i}$ as in (\ref{eq: definition of H}) with
respect to $(y_{i},z_{i})$ for $i=1,2$. Then both $(y_{1},z_{1})$
and $(y_{2},z_{2})$ satisfy equations (\ref{eq: ODE of y}) and (\ref{eq: ODE of z}),
and $H_{1}$ and $H_{2}$ satisfy equation (\ref{eq: equ. of H-2}).

To prove that $u_{1}=\la^{(p-N)/p}u_{2}(\cdot/\la)$ for some $\la>0$,
it is equivalent to prove that $y_{1}=y_{2}(\cdot-t_{0})$ for some
$t_{0}\in\R$. Up to a translation, we assume that both $y_{1}$ and
$y_{2}$ satisfy (\ref{eq: translation of y}). We prove that $y_{1}\equiv y_{2}$
on $\R$. Note that under this assumption, we have that $(y_{1}(0),z_{1}(0))=(y_{2}(0),z_{2}(0))=(M,-(\de M)^{p-1})$
by (\ref{eq: translation of y}) and (\ref{eq: z(0)}).

Define $f:(\ga_{1},\ga_{2})\to\R$ by
\[
f(\ga)=-\frac{p^{*}(s)-p}{p(p-1)}|\ga|^{2-p}\Ga_{\mu}(\ga).
\]
Then by (\ref{eq: equ. of H-2}) and (\ref{eq: range of H}), both
$H_{1}$ and $H_{2}$ are solutions to the following initial value
problem
\begin{equation}
\begin{cases}
H^{\prime}(t)=f(H(t)) & \text{for }(t,H)\in I\times(\ga_{1},\ga_{2}),\\
H(0)=\de.
\end{cases}\label{eq: equ. of H-3}
\end{equation}
In equation (\ref{eq: equ. of H-3}), $I=\R$ in the case \emph{$0\le\mu<\bar{\mu}$},
and $I=\R\backslash\{t_{-},t_{-}^{\prime}\}$ in the case \emph{$-\wq<\mu<0$},
where $t_{-}<0$ is the number defined as in Lemma \ref{lem: properties of z}
with respect to $z_{1}$ and $t_{-}^{\prime}<0$ the number with respect
to $z_{2}$. So we have two cases:

\emph{Case 1: $0\le\mu<\bar{\mu}$; }

\emph{Case 2: $-\wq<\mu<0$. }

In Case 1, we have $I=\R$ in equation (\ref{eq: equ. of H-3}). Note
that in this case, $0\le\ga_{1}<\de<\ga_{2}$. Then $f\in C^{1}(\ga_{1},\ga_{2})$.
Hence $f$ is locally Lipshitz in $(\ga_{1},\ga_{2})$. Then by Lemma
\ref{lem:Uniqueness lemma} (1), equation (\ref{eq: equ. of H-3})
admits at most one solution. Hence $H_{1}\equiv H_{2}$ on $\R.$
It follows from equation (\ref{eq:integral of y gamma}) that $y_{1}\equiv y_{2}$
on $\R$. So the uniqueness in Case 1 is proved.

In case 2, we have that $I=\R\backslash\{t_{-},t_{-}^{\prime}\}$
in equation (\ref{eq: equ. of H-3}). Note that in this case $0\in(\ga_{1},\ga_{2})$.
We divide the proof into three cases:

\emph{Case 2.1: $p=2$;}

\emph{Case 2.2: $1<p<2$;}

\emph{Case 2.3: $2<p<N$. }

In Case 2.1, $f(\ga)=-(2^{*}(s)-2)(\ga^{2}-(N-2)\ga+\mu)/2$. It is
obvious that $f\in C^{1}(\ga_{1},\ga_{2})$. Hence $f$ is locally
Lipshitz in $(\ga_{1},\ga_{2})$. So we can prove that $y_{1}\equiv y_{2}$
on $\R$ in the same way as that of Case 1. The uniqueness in Case
2.1 is proved.

In Case 2.2, $f$ is not Lipshitz in any neighborhood of $\ga=0$.
We can not use above argument. Let $y=y_{1}-y_{2}$ and $z=z_{1}-z_{2}.$
Then $y$ satisfies equation
\begin{eqnarray}
 & {\displaystyle \begin{cases}
(e^{-\delta t}y)^{\prime}=e^{-\delta t}\left(|z_{1}|^{\frac{1}{p-1}-1}z_{1}-|z_{2}|^{\frac{1}{p-1}-1}z_{2}\right) & \text{in }\R,\\
y(0)=0,
\end{cases}}\label{eq: equ. of y_1-y_2}
\end{eqnarray}
and $z$ satisfies equation
\begin{eqnarray}
 & {\displaystyle \begin{cases}
(e^{\delta t}z)^{\prime}=e^{\delta t}(y_{2}^{p^{*}(s)-1}-y_{1}^{p^{*}(s)-1}+\mu y_{2}^{p-1}-\mu y_{1}^{p-1}) & \text{in }\R,\\
z(0)=0.
\end{cases}}\label{eq: equ. of z_1 -z_2}
\end{eqnarray}
Fix a number $T$, $T>0$. Since $1<p<2$, the function $|t|^{\frac{1}{p-1}-1}t$
is continuously differentiable on $\R$. We have that
\begin{eqnarray*}
\left||z_{1}|^{\frac{1}{p-1}-1}z_{1}-|z_{2}|^{\frac{1}{p-1}-1}z_{2}\right|\leq\left(\frac{1}{p-1}\sup_{\tau\in\R}|\tau|^{\frac{2-p}{p-1}}\right)|z_{1}-z_{2}|=:C_{3}|z| &  & \text{on }[-T,T].
\end{eqnarray*}
Recall that $0<y_{i}(t)$ for all $i=1,2$. Hence $\inf_{[-T,T]}y_{1}>0$
and $\inf_{[-T,T]}y_{2}>0$. We have that
\begin{eqnarray*}
\left|y_{2}^{p^{*}(s)-1}-y_{1}^{p^{*}(s)-1}\right|+\left|\mu y_{2}^{p-1}-\mu y_{1}^{p-1}\right|\leq C_{4}|y_{1}-y_{2}|=C_{4}|y| &  & \text{on }[-T,T],
\end{eqnarray*}
where $C_{4}>0$ is a constant depending on $N,p,\mu,s$, $\inf_{[-T,T]}y_{1}$
and $\inf_{[-T,T]}y_{2}$. Let $C_{T}=\max(C_{3},C_{4})$. Then by
equation (\ref{eq: equ. of z_1 -z_2}) we obtain that
\begin{eqnarray}
e^{\delta t}|z(t)|=\left|\int_{0}^{t}(e^{\delta\tau}z(\tau))^{\prime}d\tau\right|\leq C_{T}\int_{0}^{t}e^{\delta\tau}|y(\tau)|d\tau &  & \text{for }0<t<T.\label{eq: 4.star-3}
\end{eqnarray}
Write $Y(t)=e^{-\delta t}y(t)$ for $t\in\R$. By equation (\ref{eq: equ. of y_1-y_2})
and above estimate, we obtain that
\begin{eqnarray*}
|Y(t)|\leq C_{T}^{2}t\int_{0}^{t}|Y(\tau)|d\tau &  & \text{for }0<t<T.
\end{eqnarray*}
Since $Y(0)=0$, it follows from the well known Gronwall's inequality
that $Y\equiv0$ in $[0,T]$. Hence $y\equiv0$ on $[0,T]$. We can
prove similarly that $y\equiv0$ on $[-T,0]$. Since $T>0$ is arbitrary,
we obtain that $y\equiv0$ on $\R.$ So the uniqueness is proved in
Case 2.2.

It remains to consider Case 2.3. First we prove that $t_{-}=t_{-}^{\prime}$.
With no loss of generality, we assume that $t_{-}^{\prime}\le t_{-}<0$.
Then both $z_{1}$ and $z_{2}$ do not change sign in the interval
$(t_{-},\wq)$. Precisely, both $z_{1}$ and $z_{2}$ are negative
in $(t_{-},\wq)$. Then the function $|z_{i}(t)|^{\frac{1}{p-1}-1}z_{i}(t)$
is continuously differentiable in $(t_{-},\wq)$. We can apply the
same argument as that of Case 2.2 to show that $y\equiv0$ in $(t_{-},\wq)$.
Then it follows from (\ref{eq: z determined by y and H}) that $z\equiv0$
in $(t_{-},\wq)$. In particular, we have that $z_{2}(t_{-})=z_{1}(t_{-})=0$.
Hence we apply Lemma \ref{lem: properties of z} (3) to $z_{2}$ and
obtain that $t_{-}^{\prime}=t_{-}$. Thus in case 2.3 we have that
$I=\R\backslash\{t_{-}\}$.

We still need to show that $y_{1}\equiv y_{2}$ in $(-\wq,t_{-})$.
Consider the following initial value problem
\begin{equation}
\begin{cases}
H^{\prime}(t)=f(H(t)) & \text{for }(t,H)\in(-\wq,t_{-})\times(\ga_{1},\ga_{2}),\\
H(t_{-})=0.
\end{cases}\label{eq: equ. of H-4}
\end{equation}
Then both $H_{1}$ and $H_{2}$ are nondecreasing solutions to equation
(\ref{eq: equ. of H-4}) in $(-\wq,t_{-})$. Note that $\Ga_{\mu}$
is strictly decreasing in $(\ga_{1},0)$. Hence in Case 2.3, $f$
is strictly increasing in $(\ga_{1},0)$. Then by Lemma \ref{lem:Uniqueness lemma},
equation (\ref{eq: equ. of H-4}) admits at most one nondecreasing
solution in $(-\wq,t_{-})$. Hence $H_{1}\equiv H_{2}$ in $(-\wq,t_{-})$.
It follows from (\ref{eq:integral of y gamma}) that $y_{1}\equiv y_{2}$
in $(-\wq,t_{1})$. This completes the proof for Case 2.3 and so the
proof of Theorem \ref{thm: Uniqueness up to dilation} is complete.
\end{proof}

\section{Proof of Theorem \ref{thm: Asym. beha. 1}}

Let $u\in W_{0}^{1,p}(B)$ be a positive radial solution to equation
(\ref{eq: Local Object}). Recall that we assume that $B$ is the
unit ball in $\R^{N}$ centered at the origin. Then $u$ is a solution
to the following ordinary differential equation
\begin{equation}
\begin{cases}
-\left(r^{N-1}|u'(r)|^{p-2}u'(r)\right)^{\prime}=\left({\displaystyle \frac{\mu}{r^{p}}+\frac{u(r){}^{p^{*}(s)-p}}{r^{s}}}+\la\right)u(r)^{p-1}r^{N-1}, & r\in(0,1),\\
u(r)>0, & r\in(0,1),\\
u(1)=0.
\end{cases}\label{eq: ODE of weak solutions}
\end{equation}
Since $u\in W_{0}^{1,p}(B)$, we have
\begin{equation}
\int_{0}^{1}\left(|u(r)|^{p}+|u^{\prime}(r)|^{p}\right)r^{N-1}dr=\frac{1}{\om_{N-1}}\int_{B}\left(|u|^{p}+|\na u|^{p}\right)dx<\wq,\label{eq: polar coordinate formular}
\end{equation}
where $\om_{N-1}$ is the surface measure of the unit sphere in $\R^{N}$.

Before proving Theorem \ref{thm: Asym. beha. 1}, we remark that in
fact both $u$ and $r^{N-1}|u^{\prime}|^{p-2}u^{\prime}$ are continuously
differentiable in $(0,1)$, and equation (\ref{eq: ODE of weak solutions})
can be understood in the classical sense. Indeed, it is well known
that every radially symmetric function in $W_{0}^{1,p}(B)$, after
modifying on a set of measure zero, is a continuous function in $(0,1)$.
Then by equation (\ref{eq: ODE of weak solutions}) we deduce that
$r^{N-1}|u^{\prime}|^{p-2}u^{\prime}\in C^{1}(0,1)$. Thus equation
(\ref{eq: ODE of weak solutions}) can be understood in the classical
sense.

We prove Theorem \ref{thm: Asym. beha. 1} now. We only prove Theorem
\ref{thm: Asym. beha. 1} in the case $0<\mu<\bar{\mu}$. We can prove
Theorem \ref{thm: Asym. beha. 1} in the case $\mu\le0$ similarly.
In the case when $0<\mu<\bar{\mu}$, the same result was obtain by
the authors \cite{He-Xiang2015} for positive radial weak solutions
to the following equation
\begin{eqnarray*}
-\Delta_{p}u-\frac{\mu}{|x|^{p}}|u|^{p-2}u=f(u)-m|u|^{p-2}u, &  & \text{in }\R^{N},
\end{eqnarray*}
where $f$ satisfies the growth condition $|f(t)|\le C(|t|^{p-1}+|t|^{p^{*}-1})$
for all $t\in\R$ by the assumptions there. Theorem \ref{thm: Asym. beha. 1}
can be proved by the same argument as that of \cite[Theorem 1.1]{He-Xiang2015}.
For the sake of completeness, we mimic a proof here.

\begin{proof}[Proof of Theorem \ref{thm: Asym. beha. 1}] Let $u\in W_{0}^{1,p}(B)$
be a positive radial solution to equation (\ref{eq: Local Object})
with \emph{$0<\mu<\bar{\mu}$} in the following. To start with we
claim that
\begin{eqnarray}
u^{\prime}(r)<0 &  & \text{ for }r\text{ sufficiently small. }\label{eq: decreasing property}
\end{eqnarray}
Indeed, note that since $u\in W_{0}^{1,p}(B)$ is a radial function,
we have by \cite[Corollary II.1]{Lions1982} that
\begin{eqnarray*}
u(r)r^{\frac{N-p}{p}}=o(1) &  & \text{as }r\to0.
\end{eqnarray*}
It follows that
\begin{eqnarray}
u(r){}^{p^{*}(s)-p}r^{p-s}=o(1) &  & \text{as }r\rightarrow0.\label{eq: estimate of f over u}
\end{eqnarray}
Hence
\begin{eqnarray*}
\left({\displaystyle \frac{\mu}{r^{p}}+\frac{u(r){}^{p^{*}(s)-p}}{r^{s}}}+\la\right)u(r)^{p-1}r^{N-1}>\frac{\mu}{2}u(r)^{p-1}r^{N-p-1}>0 &  & \mbox{for }r\text{ small enough}.
\end{eqnarray*}
Therefore $\left(r^{N-1}|u^{\prime}|^{p-2}u^{\prime}\right)^{\prime}<0$
for $r$ small enough by equation (\ref{eq: ODE of weak solutions}).
Hence $r^{N-1}|u^{\prime}|^{p-2}u^{\prime}$ is strictly decreasing
for $r$ small enough. So we can assume that $\lim_{r\to0}r^{N-1}|u^{\prime}|^{p-2}u^{\prime}=a$
for some $a\in(-\wq,\wq]$. We will prove that $a=0$. Suppose, on
the contrary, that $a\ne0$. Then there exist constants $C,r_{0}>0$
such that $|u^{\prime}(r)|\ge Cr^{-\frac{N-1}{p-1}}$ for $0<r<r_{0}$.
Then we have
\[
\int_{0}^{r_{0}}|u^{\prime}(r)|^{p}r^{N-1}dr\ge C\int_{0}^{r_{0}}r^{-\frac{N-1}{p-1}}dr=\wq.
\]
We reach a contradiction to (\ref{eq: polar coordinate formular}).
Hence $a=0$. Therefore $r^{N-1}|u^{\prime}|^{p-2}u^{\prime}<0$ for
$r$ small enough. This proves (\ref{eq: decreasing property}).

Consider the function
\begin{eqnarray}
w(r)=-\frac{r^{p-1}|u^{\prime}(r)|^{p-2}u^{\prime}(r)}{u^{p-1}(r)} &  & \text{for }r>0.\label{eq: definition of w}
\end{eqnarray}
Then $w\in C^{1}(0,1)$, $w(r)>0$ for $r>0$ small enough by (\ref{eq: decreasing property}),
and $w$ satisfies
\begin{equation}
w^{\prime}(r)=\frac{1}{r}\left(\Gamma_{\mu}\Big(w(r)^{\frac{1}{p-1}}\Big)+u(r){}^{p^{*}(s)-p}r^{p-s}+\la r^{p}\right).\label{eq:equa.of w}
\end{equation}
Recall that $\Gamma_{\mu}$ is defined as in (\ref{eq: definition of Gamma-mu}).
To prove Theorem \ref{thm: Asym. beha. 1}, it is enough to prove
that
\begin{eqnarray}
w(r)=\ga_{1}^{p-1}+o(r^{\de}) &  & \text{as }r\to0,\label{eq: estimate of w}
\end{eqnarray}
for some $\de\in(0,1)$. Here $\ga_{1}$ is defined as in (\ref{eq: definition of ga-1 and ga-2}).
In the case when $0<\mu<\bar{\mu},$ we note that $0<\ga_{1}<(N-p)/p$.

First, we prove that $\lim_{r\rightarrow0}w(r)$ exists and
\begin{equation}
\lim_{r\rightarrow0}w(r)=\gamma_{1}^{p-1}.\label{eq: limit of w at zero}
\end{equation}
To prove that $\lim_{r\rightarrow0}w(r)$ exists, we suppose, on the
contrary, that
\[
\beta\equiv\limsup_{r\rightarrow0}w(r)>\liminf_{r\rightarrow0}w(r)\equiv\alpha.
\]
Then there exist two sequences of positive numbers $\{\xi_{i}\}$
and $\{\eta_{i}\}$ such that $\xi_{i}\rightarrow0$ and $\eta_{i}\rightarrow0$
and that $\eta_{i}>\xi_{i}>\eta_{i+1}$ for all $i=1,2,\cdots$. Moreover,
the function $w$ has a local maximum at $\xi_{i}$ and a local minimum
at $\eta_{i}$ for all $i=1,2,\cdots$, and
\begin{eqnarray*}
\lim_{i\to\wq}w(\xi_{i})=\beta, &  & \lim_{i\to\wq}w(\eta_{i})=\alpha.
\end{eqnarray*}
Note that $w^{\prime}(\xi_{i})=w^{\prime}(\eta_{i})=0$. By equation
(\ref{eq:equa.of w}), we have that
\[
\Gamma_{\mu}\Big(w^{\frac{1}{p-1}}(\xi_{i})\Big)+\la\xi_{i}^{p}+u(\xi_{i}){}^{p^{*}(s)-p}\xi_{i}^{p-s}=0,
\]
and that
\[
\Gamma_{\mu}\Big(w^{\frac{1}{p-1}}(\eta_{i})\Big)+\la\eta_{i}^{p}+u(\eta_{i}){}^{p^{*}(s)-p}\eta_{i}^{p-s}=0.
\]
By (\ref{eq: estimate of f over u}) and the above two equalities,
we have that
\[
\lim_{i\rightarrow\infty}\Gamma_{\mu}\Big(w^{\frac{1}{p-1}}(\xi_{i})\Big)=\lim_{i\rightarrow\infty}\Gamma_{\mu}\Big(w^{\frac{1}{p-1}}(\eta_{i})\Big)=0.
\]
Since $\Gamma_{\mu}(s)\rightarrow\infty$ as $|s|\rightarrow\infty$,
$\{w(\xi_{i})\}$ and $\{w(\eta_{i})\}$ are bounded. So $\al,\be$
are finite and
\[
\Gamma_{\mu}\big(\beta^{\frac{1}{p-1}}\big)=\Gamma_{\mu}\big(\alpha^{\frac{1}{p-1}}\big)=0.
\]
Recall that $\Gamma_{\mu}(\ga)=0$ if and only if $\ga=\ga_{1}$ or
$\ga=\ga_{2}$. Recall also that in the case when $0<\mu<\bar{\mu},$
we have $0<\ga_{1}<(N-p)/p<\ga_{2}<(N-p)/(p-1)$. Hence
\begin{eqnarray*}
\beta=\gamma_{2}^{p-1} & \text{and} & \al=\ga_{1}^{p-1}.
\end{eqnarray*}
That is,
\begin{eqnarray*}
\lim_{i\to\wq}w(\xi_{i})=\gamma_{2}^{p-1} & \text{ and } & \lim_{i\to\wq}w(\eta_{i})=\gamma_{1}^{p-1}.
\end{eqnarray*}
Note that $\ga_{1}<(N-p)/p<\ga_{2}$. So there exists $\zeta_{i}\in(\eta_{i+1},\xi_{i})$
such that
\[
w(\eta_{i+1})<w(\zeta_{i})=\left(\frac{N-p}{p}\right)^{p-1}<w(\xi_{i})
\]
for $i$ large enough. Then by (\ref{eq: estimate of f over u}) and
equation (\ref{eq:equa.of w}), we obtain that
\[
\zeta_{i}w^{\prime}(\zeta_{i})=\Gamma_{\mu}\left(\frac{N-p}{p}\right)+\la\zeta_{i}^{p}+u(\zeta_{i}){}^{p^{*}(s)-p}\zeta_{i}^{p-s}=-(\bar{\mu}-\mu)+o(1)<0
\]
for $i$ large enough. Here we used the fact that
\[
\Gamma_{\mu}\left(\frac{N-p}{p}\right)=-(\bar{\mu}-\mu).
\]
Hence $w^{\prime}(\zeta_{i})<0$ for $i$ large enough. Therefore
$w$ is strictly decreasing in a neighborhood of $\zeta_{i}$. Since
$\zeta_{i}<\xi_{i}$ and $w(\zeta_{i})<w(\xi_{i})$, there exists
$\zeta_{i}<\zeta_{i}^{\prime}<\xi_{i}$ such that $w(r)\le w(\zeta_{i})$
for $\zeta_{i}<r<\zeta_{i}^{\prime}$ and $w(\zeta_{i}^{\prime})=w(\zeta_{i})$.
Thus $w^{\prime}(\zeta_{i}^{\prime})\ge0$. However, by equation (\ref{eq:equa.of w}),
we have that $w^{\prime}(\zeta_{i}^{\prime})<0$. We reach a contradiction.
Therefore $\lim_{r\rightarrow0}w(r)$ exists.

Set $k^{p-1}=\lim_{r\rightarrow0}w(r)$. We will prove that $k=\ga_{1}$.

We claim that $k\le(N-p)/p$. Otherwise, choose $\ep>0$ such that
$k-\ep>(N-p)/p$. Then for $r$ small enough we have $w(r)>(k-\ep)^{p-1}$,
that is, $-ru^{\prime}(r)/u(r)>k-\ep$ for $r$ small enough. This
implies that $u(r)\ge Cr^{\ep-k}$ for $r$ small enough, which implies
$u\not\in L^{p^{*}}(B)$. We reach a contradiction. Thus $k\le(N-p)/p$.

By (\ref{eq: estimate of f over u}) and equation (\ref{eq:equa.of w}),
we have that
\[
\lim_{r\to0}rw^{\prime}(r)=\Ga_{\mu}(k).
\]
We claim that $\Ga_{\mu}(k)=0$. Otherwise, suppose that $\Ga_{\mu}(k)\ne0$.
Note that for any $0<s<s_{0}$, we have
\[
w(s_{0})=w(s)+\int_{s}^{s_{0}}w^{\prime}(t)dt.
\]
Then $\Ga_{\mu}(k)\ne0$ implies that $\lim_{s\to0}\left|\int_{s}^{s_{0}}w^{\prime}(t)dt\right|=\wq$
if $s_{0}$ is small enough. This contradicts to (\ref{eq: limit of w at zero}).
Hence $\Ga_{\mu}(k)=0$. Recall that $\Ga_{\mu}(\ga)=0$ if and only
if $\ga=\ga_{1}$ or $\ga=\ga_{2}$. Thus we have either $k=\ga_{1}$
or $k=\ga_{2}$. Then we deduce that $k=\ga_{1}$ since $k\le(N-p)/p<\ga_{2}$.
This proves (\ref{eq: limit of w at zero}).

As a result, (\ref{eq: limit of w at zero}) implies that for any
$\ep>0$ sufficiently small there exist $C,c>0$ such that
\[
cr^{-\ga_{1}+\ep}\le u(r)\le Cr^{-\ga_{1}-\ep}
\]
for $r>0$ small enough. Choose $\ep=\ep_{0}>0$ such that $p-s-(p^{*}(s)-p)(\ga_{1}+\ep_{0})>0$.
This is possible since $\ga_{1}<(N-p)/p$. Applying (\ref{eq: estimate of f over u}),
we obtain that
\begin{equation}
u(r){}^{p^{*}(s)-p}r^{p-s}\le Cr^{p-s-(p^{*}(s)-p)(\ga_{1}+\ep_{0})}\equiv Cr^{\de_{0}}\label{eq: 3.1.8}
\end{equation}
for $r>0$ small enough. Here $\de_{0}=p-s-(p^{*}(s)-p)(\ga_{1}+\ep_{0})>0$.

Now we prove (\ref{eq: estimate of w}). Let $w_{1}(r)=w(r)-\ga_{1}^{p-1}$.
Then $w_{1}(r)\to0$ as $r\to0$. We prove that $w_{1}(r)=o(r^{\de})$
as $r\to0$ for some $\de>0$.

By equation (\ref{eq:equa.of w}) and the definition of $\Ga_{\mu}$
(see (\ref{eq: definition of Gamma-mu})), we have
\begin{equation}
\begin{aligned}w_{1}^{\prime}(r) & =w^{\prime}(r)=\frac{1}{r}\Gamma_{\mu}\left(w^{\frac{1}{p-1}}(r)\right)+\frac{1}{r}\left(u(r){}^{p^{*}(s)-p}r^{p-s}+\la r^{p}\right)\\
 & =\frac{1}{r}\left((p-1)w^{\frac{p}{p-1}}(r)-(N-p)w(r)+\mu\right)+\frac{1}{r}\left(u(r){}^{p^{*}(s)-p}r^{p-s}+\la r^{p}\right)\\
 & =\frac{A(r)}{r}w_{1}(r)+B(r),
\end{aligned}
\label{eq: equation of w_1}
\end{equation}
for $r$ small enough, where $A(r)\to p\ga_{1}-(N-p)<0$ as $r\to0$
and
\begin{equation}
B(r)=\frac{1}{r}\left(u(r){}^{p^{*}(s)-p}r^{p-s}+\la r^{p}\right)=O\left(r^{\de_{0}-1}\right)\qquad\text{ as }r\to0,\label{eq: function B}
\end{equation}
by (\ref{eq: 3.1.8}). Here $\de_{0}>0$ is defined as in (\ref{eq: 3.1.8}).

Fix $r_{0}>0$ small and define $h(r)=\int_{r}^{r_{0}}A(\tau)\tau^{-1}d\tau$
for $0<r<r_{0}$. Since $w_{1}$ is a solution to equation (\ref{eq: equation of w_1}),
it has the following form
\[
w_{1}(r)=\int_{0}^{r}e^{h(s)-h(r)}B(s)ds.
\]
Since $h(s)-h(r)=\int_{s}^{r}A(\tau)\tau^{-1}d\tau<0$ for $0<s<r$,
we obtain that $e^{h(s)-h(r)}\le1$ for $0<s<r$. Hence by (\ref{eq: function B}),
we have for $r$ small enough that
\[
|w_{1}(r)|\le\int_{0}^{r}|B(s)|ds\le Cr^{\de_{0}}.
\]
Here $\de_{0}>0$ is as in (\ref{eq: 3.1.8}). This proves (\ref{eq: estimate of w}).

Recall that $w$ is defined as (\ref{eq: definition of w}). The conclusion
of Theorem \ref{thm: Asym. beha. 1} follows easily from estimate
(\ref{eq: estimate of w}). The proof of Theorem \ref{thm: Asym. beha. 1}
in the case  $0<\mu<\bar{\mu}$ is complete. \end{proof}

\appendix

\section{A comparison result}

Let $\Om$ be a bounded domain in $\R^{N}$ containing the origin.
Define the operator $-L_{p}$ by
\begin{eqnarray*}
-L_{p}u=-\De_{p}u-\frac{\mu}{|x|^{p}}|u|^{p-2}u, &  & u\in W_{0}^{1,p}(\Om).
\end{eqnarray*}
We have the following result.
\begin{lem}
\label{lem: eigenvalue comparison} Let $\rho_{1},\rho_{2}$ be two
nonnegative functions in $L^{\frac{N}{p}}(\Om)$. Denote by $\la_{1}(\rho_{i})$
the first eigenvalue of the operator $-L_{p}$ with respect to weight
$\rho_{i}$, $i=1,2$, that is,
\[
\la_{1}(\rho_{i})=\inf\left\{ \frac{Q(\var)}{\int_{\Om}\rho_{i}|\var|^{p}dx}:\var\in W_{0}^{1,p}(\Om),\var\ne0\right\} ,
\]
for $i=1,2$, where
\[
Q(\var)=\int_{\Om}\left(|\na\var|^{p}-\frac{\mu}{|x|^{p}}|\var|^{p}\right)dx.
\]
If $\rho_{1}\ge\rho_{2}$, then either $\la_{1}(\rho_{1})<\la_{1}(\rho_{2})$
or $\la_{1}(\rho_{1})=\la_{1}(\rho_{2})$ and $e_{i}=0$ $(i=1,2)$
on $\{x\in\Om:\rho_{1}(x)\ne\rho_{2}(x)\}$, where $e_{1},e_{2}$
are the first eigenfunctions corresponding to the first eigenvalues
$\la_{1}(\rho_{1})$ and $\la_{1}(\rho_{2})$ respectively. \end{lem}
\begin{proof}
In the case when $\mu=0$, Lemma \ref{lem: eigenvalue comparison}
was proved by Adimurthi and Yadava \cite[Lemma 4.1]{AdimurthiYadava1994}.
Their argument can be easily applied to prove Lemma \ref{lem: eigenvalue comparison}.
For completeness, we give a proof here.

It is direct to verify that $\la_{1}(\rho_{1})\le\la_{1}(\rho_{2})$
by definition. Suppose that $\la_{1}(\rho_{1})=\la_{1}(\rho_{2})$
and $e_{1},e_{2}$ are the first eigenfunctions corresponding to the
first eigenvalues $\la_{1}(\rho_{1})$ and $\la_{1}(\rho_{2})$ respectively.
Then $e_{i}$, $i=1,2$, are nonpositive or nonnegative functions
in $\Om$. We assume that $e_{i}\ge0$ for both $i=1,2$. Then
\[
\la_{1}(\rho_{i})=\frac{Q(e_{i})}{\int_{\Om}\rho_{i}|e_{i}|^{p}dx},\quad i=1,2.
\]
Therefore,
\[
\frac{Q(e_{2})}{\int_{\Om}\rho_{1}|e_{2}|^{p}dx}\ge\frac{Q(e_{1})}{\int_{\Om}\rho_{1}|e_{1}|^{p}dx}=\la_{1}(\rho_{1})=\la_{1}(\rho_{2})=\frac{Q(e_{2})}{\int_{\Om}\rho_{2}|e_{2}|^{p}dx}.
\]
Since $\rho_{1}\ge\rho_{2}$, we obtain that $\int_{\Om}\rho_{1}|e_{1}|^{p}dx=\int_{\Om}\rho_{2}|e_{1}|^{p}dx$.
That is,
\[
\int_{\Om}(\rho_{1}-\rho_{2})|e_{1}|^{p}dx=0.
\]
Since $e_{1}$ is nonnegative in $\Om$, we have that $e_{1}=0$ on
$\{x\in\Om:\rho_{1}(x)\ne\rho_{2}(x)\}$.

Hence
\[
\la_{1}(\rho_{2})=\la_{1}(\rho_{1})=\frac{Q(e_{1})}{\int_{\Om}\rho_{1}|e_{1}|^{p}dx}=\frac{Q(e_{1})}{\int_{\Om}\rho_{2}|e_{1}|^{p}dx},
\]
which implies that $e_{1}$ is also an eigenfunction of $\la_{1}(\rho_{2})$.
Thus $e_{1}=ke_{2}$ for some $k\ne0$ (see \cite{Zhang-Wang-Liu2013}).
Thus $e_{2}=0$ on $\{x\in\Om:\rho_{1}(x)\ne\rho_{2}(x)\}$. This
finishes the proof of Lemma \ref{lem: eigenvalue comparison}.
\end{proof}

\section{A uniqueness result on ordinary differential equations}

The following result can be found in standard textbooks on ordinary
differential equations.
\begin{lem}
\label{lem:Uniqueness lemma}Let $(a,b)\subset\R$ and $(c,d)\subset\R$
be two intervals. Assume that $f:\R\rightarrow\R$ be a continuous
function. Consider the initial value problem
\begin{equation}
\begin{cases}
y^{\prime}(t)=f(y(t)), & (t,y)\in(a,b)\times(c,d),\\
y(t_{0})=y_{0},
\end{cases}\label{eq:uniqueness ODE}
\end{equation}
 for some $(t_{0},y_{0})\in(a,b)\times(c,d)$. Then we have

(1) if $f$ is locally Lipshcitz continuous in $(c,d)$, equation
(\ref{eq:uniqueness ODE}) admits at most one solution on $(a,b)$;

(2) if $f$ is nonincreasing in $(y_{0},d)$, then equation (\ref{eq:uniqueness ODE})
admits at most one nondecreasing solution on $(t_{0},b)$;

(3) if $f$ is nondecreasing in $(c,y_{0})$, then equation (\ref{eq:uniqueness ODE})
admits at most one nondecreasing solution on $(a,t_{0})$. \end{lem}
\begin{proof}
(1) can be proved in a standard way. We omit the details. We only
prove conclusion (2). We can prove (3) similarly.

Suppose that $f$ is nonincreasing in $(y_{0},d)$ and $y_{1},y_{2}$
are two distinct nondecreasing solutions of equation (\ref{eq:uniqueness ODE})
on $(t_{0},b)$. With no loss of generality, we assume that $y_{1}(t_{1})>y_{2}(t_{1})$
for some $t_{1}\in(t_{0},b)$. Let
\[
t_{2}=\inf\{t\in[t_{0},t_{1}):y_{1}(s)>y_{2}(s)\text{ for }s\in(t,t_{1})\}.
\]
Then $t_{1}>t_{2}\geq t_{0}$, $y_{1}(t_{2})=y_{2}(t_{2})$, and $y_{1}(t)>y_{2}(t)$
for $t\in(t_{2},t_{1}]$. Hence
\begin{eqnarray*}
y_{1}^{\prime}(t)-y_{2}^{\prime}(t)=f(y_{1}(t))-f(y_{2}(t))\leq0 &  & \text{for }t\in(t_{2},t_{1}),
\end{eqnarray*}
since $f$ is nonincreasing in $(y_{0},d)$. Thus $y_{1}-y_{2}$ is
nonincreasing on $[t_{2},t_{1}]$. In particular, we have that $y_{1}(t_{1})-y_{2}(t_{1})\leq y_{1}(t_{2})-y_{2}(t_{2})=0$.
We reach a contradiction. This proves (2).
\end{proof}
\emph{Acknowledgment. }The second named author is financially supported
by the Academy of Finland, project 259224.

\end{document}